\documentclass[10pt]{amsart}
\usepackage[margin=1in]{geometry}
\usepackage{amssymb}
\usepackage{mathrsfs}
\usepackage{mathabx,epsfig}
\usepackage{cite}
\usepackage{amscd}
\usepackage{amsmath}
\usepackage{algorithmicx}
\usepackage{algpseudocode}
\usepackage{stmaryrd}
\usepackage{enumerate}
\usepackage{pict2e} 
\usepackage[all]{xy} 
\usepackage[usenames,dvipsnames]{color}

\usepackage[sc]{mathpazo}
\usepackage[T1]{fontenc}

\usepackage{hyperref}

\newtheorem{thm}{Theorem}
\newtheorem{pro}[thm]{Proposition}
\newtheorem{lmm}[thm]{Lemma}
\newtheorem{cor}[thm]{Corollary}
\newtheorem{ex}[thm]{Example}
{\theoremstyle{definition} 
\newtheorem{dfn}[thm]{Definition}
\newtheorem{rmk}[thm]{Remark}}

\newcommand{\N}{\mathbb{N}}
\newcommand{\Z}{\mathbb{Z}}
\newcommand{\A}{\mathbb{A}}
\renewcommand{\P}{\mathbb{P}}
\newcommand{\R}{\mathbb{R}}
\renewcommand{\O}{\mathcal{O}}

\newcommand{\TT}{\mathbb{T}}
\renewcommand{\tt}{{\bf t}}

\newcommand{\V}{\mathbb{V}}

\renewcommand{\L}{\mathscr{L}}
\renewcommand{\H}{\mathscr{H}}

\newcommand{\CC}{\mathbb{C}}

\newcommand{\Spec}{{\rm Spec}\,}
\newcommand{\supp}{{\rm supp}}
\newcommand{\Proj}{{\rm Proj}\,}
\newcommand{\st}{\mathsf{st}}

\newcommand{\IN}{\mathsf{in}}

\newcommand{\punc}{{\rm punc}}
\newcommand{\lin}{{\rm lin}}
\newcommand{\lex}{{\rm lex}}

\newcommand{\BB}{{\rm BB}}

\renewcommand\H{{\bf{H}}}
\def\acts{\mathrel{\reflectbox{$\righttoleftarrow$}}}

\def\pone{{
    \begin{picture}(8,8)
    \multiput(1,0)(0,6){2}{\line(1,0){6}}
    \multiput(1,0)(6,0){2}{\line(0,1){6}}
    \end{picture}\hspace{.02cm}
    }}

\def\ptwo{{
    \begin{picture}(14,8)
    \multiput(1,0)(0,6){2}{\line(1,0){12}}
    \multiput(1,0)(6,0){3}{\line(0,1){6}}
    \end{picture}\hspace{.02cm}
    }}

\def\pthree{{
    \begin{picture}(20,8)
    \multiput(1,0)(0,6){2}{\line(1,0){18}}
    \multiput(1,0)(6,0){4}{\line(0,1){6}}
    \end{picture}\hspace{.02cm}
    }}

\def\pfour{{
    \begin{picture}(26,8)
    \multiput(1,0)(0,6){2}{\line(1,0){24}}
    \multiput(1,0)(6,0){5}{\line(0,1){6}}
    \end{picture}\hspace{.02cm}
    }}

\def\poneone{{
    \begin{picture}(8,14)
    \multiput(1,0)(0,6){3}{\line(1,0){6}}
    \multiput(1,0)(6,0){2}{\line(0,1){12}}
    \end{picture}\hspace{.02cm}
    }}

\def\poneoneone{{
    \begin{picture}(8,20)
    \multiput(1,0)(0,6){4}{\line(1,0){6}}
    \multiput(1,0)(6,0){2}{\line(0,1){18}}
    \end{picture}\hspace{.02cm}
    }}

\def\ptwooneone{{
    \begin{picture}(14,20)
    \multiput(1,0)(0,6){2}{\line(1,0){12}}
    \multiput(1,12)(0,6){2}{\line(1,0){6}}
    \multiput(1,0)(6,0){2}{\line(0,1){18}}
    \put(13,0){\line(0,1){6}}
    \end{picture}\hspace{.02cm}
    }}

\def\ptwotwo{{
    \begin{picture}(14,14)
    \multiput(1,0)(0,6){3}{\line(1,0){12}}
    \multiput(1,0)(6,0){3}{\line(0,1){12}}
    \end{picture}\hspace{.02cm}
    }}

\def\pthreeone{{
    \begin{picture}(20,14)
    \multiput(1,0)(0,6){2}{\line(1,0){18}}
    \put(1,12){\line(1,0){6}}
    \multiput(1,0)(6,0){2}{\line(0,1){12}}
    \multiput(13,0)(6,0){2}{\line(0,1){6}}
    \end{picture}\hspace{.02cm}
    }}

\newcommand\defn[1]{{\bf #1}}

\begin{document} 

\title{On the cup product for Hilbert schemes of points in the plane}
\author{Mathias Lederer}
\email{mathias.lederer@uibk.ac.at}
\address{Department of Mathematics \\ 
University of Innsbruck \\ 
Technikerstrasse 21a \\
 A-6020 Innsbruck \\ 
 Austria}
\thanks{The author was partially supported by a Marie Curie International Outgoing Fellowship 
of the EU Seventh Framework Program}
\date{\today}

\begin{abstract} 
  We revisit Ellingsrud and Str\o mme's cellular decomposition of the Hilbert scheme of points in the projective plane. 
  We study the product of cohomology classes defined by the closures of cells, 
  deriving necessary conditions for the non-vanishing of cohomology classes. 
  Though our conditions are formulated in purely combinatorial terms, 
  the machinery for deriving them includes techniques from Bia\l ynicki-Birula theory: 
  We study closures of Bia\l ynicki-Birula cells in complete complex varieties equipped with ample line bundles. 
  We prove a necessary condition for two such closures to meet, and apply this criterion in our setting. 
\end{abstract}

\maketitle


\section{Introduction}\label{sec:intro}

The \defn{Hilbert scheme of $n$ points in the projective plane} 
is the moduli space of homogeneous ideals in $S := \CC[x_0,x_1,x_2]$, 
or equivalently, closed subschemes of $\P^2$, with constant Hilbert function $n$. 
The cellular decomposition $\P^2 = \A^2 \coprod \A^1 \coprod \A^0$ induces a decomposition into locally closed subschemes 
\[
  H^n(\P^2) = \coprod_{n_2 + n_1 + n_0 = n} H^{n_2}(\A^2) \times H^{n_1,\lin}(\A^2) \times H^{n_0,\punc}(\A^2) ,
\]
where the three factors parametrize ideals supported in the respective cofactors of $\P^2$.
Upon using the coordinate ring $S' := \CC[y_1,y_2]$ of $\A^2$
and identifying $\A^1 = \V(y_2)$ and $\A^0 = \V(y_1,y_2)$, the three factors appearing in the displayed coproduct read 
\begin{equation*}
  \begin{split}
    H^{n_2}(\A^2) & := \bigl\{ \text{ideals } I \subseteq S' : \dim(S' / I) = n_2 \bigr\} , \\
    H^{n_1,\lin}(\A^2) & := \bigl\{ \text{ideals } I \subseteq S' : \dim(S' / I) = n_1, \, \supp(I) \subseteq \V(y_2) \bigr\} , \\
    H^{n_0,\punc}(\A^2) & := \bigl\{ \text{ideals } I \subseteq S' : \dim(S' / I) = n_0, \, \supp(I) = \V(y_1,y_2) \bigr\} .
  \end{split}
\end{equation*}
They come in a chain of closed immersions 
\[
  H^{n,\punc}(\A^2) \subseteq H^{n,\lin}(\A^2) \subseteq H^n(\A^2) .
\]
The smallest member of the chain is called the \defn{punctual Hilbert scheme}, 
the largest the \defn{Hilbert scheme of points in the affine plane}. 
The scheme in the middle doesn't have a distinguished name, 
even though it is of utmost importance in linking Hilbert schemes of points to the ring of symmetric functions 
\cite[Corollary 9.15]{nakajima}. 

The scheme $H^n(\P^2)$ is smooth and projective of dimension $2n$ \cite{Fogarty_smoothness}. 
Ellingsrud and Str\o mme  
constructed cellular decompositions of the three factors $H^{n_2}(\A^2)$, $H^{n_1,\lin}(\A^2)$ and $H^{n_0,\punc}(\A^2)$, 
thus refining the above-displayed coproduct into a cellular decomposition \cite{esBetti, esCells}.
For doing so, they used specific actions of the torus $\CC^\star$ on 
$H^{n_2}(\A^2)$, $H^{n_1,\lin}(\A^2)$ and $H^{n_0,\punc}(\A^2)$, respectively.  
The fixed points of all three actions are monomial ideals $M_\Delta$. 
Here the subscript $\Delta$ is the \defn{standard set} or \defn{staircase} of the monomial ideal $M_\Delta$, 
i.e., the set of elements of $\N^2$ not showing up as exponents in the monomial ideal. 
Thus in particular, $|\Delta| = \dim_\CC S' / I$. 
Since the fixed points of the action are isolated, the \defn{Bia\l ynicki-Birula sinks}, or \defn{BB sinks}
\begin{equation*}
  \begin{split}
    H^{\Delta_2}_\lex(\A^2) & := \bigl\{ I \in H^{n_2}(\A^2) : \lim_{t \to 0} t.I = M_{\Delta_2} \bigr\} , \\
    H^{\Delta_1,\lin}_\lex(\A^2) & := \bigl\{ I \in H^{n_1,\lin}(\A^2) : \lim_{t \to 0} t.I = M_{\Delta_1} \bigr\} , \\
    H^{\Delta_0,\punc}_\lex(\A^2) & := \bigl\{ I \in H^{n_0,\punc}(\A^2) : \lim_{t \to 0} t.I = M_{\Delta_0} \bigr\} 
  \end{split}
\end{equation*}
are affine spaces \cite{bialynickiBirula}. 
Our notation is motivated by the fact that Ellingsrud and Str\o mme's choice of torus actions implies that 
the BB sinks are the schemes parametrizing ideals whose lexicographic Gr\"obner deformations are the given monomial ideals, 
\begin{equation*}
  \begin{split}
    H^{\Delta_2}_\lex(\A^2) & = \bigl\{ I \in H^{n_2}(\A^2) : \IN_\lex(I) = M_{\Delta_2} \bigr\} , \\
    H^{\Delta_1,\lin}_\lex(\A^2) & = \bigl\{ I \in H^{n_1,\lin}(\A^2) : \IN_\lex(I) = M_{\Delta_1} \bigr\} , \\
    H^{\Delta_0,\punc}_\lex(\A^2) & = \bigl\{ I \in H^{n_0,\punc}(\A^2) : \IN_\lex(I) = M_{\Delta_0} \bigr\} . 
  \end{split}
\end{equation*}
The three schemes are therefore also known as \defn{Gr\"obner basins}. 

Ellingsrud and Str\o mme, with a later correction by Huibregtse \cite{huibregtseEllingsrud}, 
were able to determine the dimensions of the three Gr\"obner basins displayed above. 
They proved that 
\begin{equation*}
  \begin{split}
    \dim(H^{\Delta}_\lex(\A^2)) & = |\Delta| + h(\Delta) , \\
    \dim(H^{\Delta,\lin}_\lex(\A^2)) & = |\Delta| , \\
    \dim(H^{\Delta,\punc}_\lex(\A^2)) & = |\Delta| - w(\Delta) ,
  \end{split}
\end{equation*}
where $h(\Delta)$ is the \defn{height} and $w(\Delta)$ is the \defn{width} of $\Delta$. 
Conca and Valla proved the same formul\ae\ using Hilbert-Burch matrices \cite{hilbertBurchMatrices}. 
The cellular decomposition of the Hilbert scheme of points in the projective plane thus reads 
\[
  H^n(\P^2) = \coprod_{|\Delta_2| + |\Delta_1| + |\Delta_0| = n} 
  (\Delta_2, \Delta_1, \Delta_0)^\circ , 
\]
where we use the shorthand notation 
\[
  (\Delta_2, \Delta_1, \Delta_0)^\circ := H^{\Delta_2}_\lex(\A^2) \times H^{\Delta_1,\lin}_\lex(\A^2) \times H^{\Delta_0,\punc}_\lex(\A^2)
\]
for the affine cell corresponding to torus fixed point $(M_{\Delta_2}, M_{\Delta_1}, M_{\Delta_0})$. 
Moreover, ignoring any danger of confusion with triples of standard sets, we denote by 
\[
  (\Delta_2, \Delta_1, \Delta_0) := \overline{(\Delta_2, \Delta_1, \Delta_0)^\circ} , 
\]
the corresponding closed variety, and by 
\[
  [\Delta_2, \Delta_1, \Delta_0] := [(\Delta_2, \Delta_1, \Delta_0)] 
\]
its cohomology class. 
Since the Hilbert scheme of points in the projective plane is smooth and projective, 
the classes $[\Delta_2, \Delta_1, \Delta_0]$, for all triples of standard sets such that $|\Delta_2| + |\Delta_1| + |\Delta_0| = n$, 
form a basis of its cohomology module $\H^\star(H^n(\P^2))$ \cite[Example 1.9.1]{fultonIntersectionTheory}. 
Ellingsrud and Str\o mme's formul\ae\ for the dimension of the affine cells of $H^n(\P^2)$ 
completely determine the additive structure of $\H^\star(H^n(\P^2))$: 
For $k=0,\ldots,2n$, a basis of $\H^k(H^n(\P^2))$, the $k$-th graded piece of cohomology\footnote{%
Since cohomology vanishes in odd degrees, we skip the factor 2 appearing in the actual values of $\star$.
We might as well be working in the Chow ring, which is isomorphic to the cohomology ring, 
the degree $k$ part in Chow corresponding to the degree $2k$ part in cohomology. 
} 
of $H^n(\P^2)$ given by 
\[
  T^k := \bigl\{ (\Delta_2, \Delta_1, \Delta_0) : |\Delta_2| + h(\Delta_2) + |\Delta_1| + |\Delta_0| - w(\Delta_0) = k \bigr\} . 
\]
The goal of the present paper is to shed some light on the multiplicative structure of $\H^\star(H^n(\P^2))$. 
We will show that it rarely happens that $[\Delta_2, \Delta_1, \Delta_0] \cdot [\Delta'_2, \Delta'_1, \Delta'_0] \neq 0$, 
making the matrix of multiplication in $\H^\star(H^n(\P^2))$ with respect to Ellingsrud and Str\o mme's basis rather sparse. 
The crucial notion is that of a family of partial orders, given by integral weight vectors. 

\begin{dfn}
\label{dfn:orderings}
  \begin{enumerate}[(a)]
  \item Let $\st_{3,n}$ denote the set of triples of standard sets whose cardinalities sum to $n$. 
  
  \item For each general enough\footnote{%
  Sufficient genericity of $\lambda$ holds if the only fixed points of the action on $H^n(\A^2)$ with weight $\lambda$ 
  are monomial ideals. This holds true if $\langle \lambda, \alpha - \beta \rangle \neq 0$ 
  for all $\alpha,\beta$ lying in any standard set of cardinality $2n+1$, say. 
  }
  weight vector $\lambda \in \Z^2$ such that $\lambda_1 < 0 < \lambda_2$, 
  we define a partial ordering $\leq_\lambda$ on $\st_{3,n}$ 
  by saying that $(\Delta_2, \Delta_1, \Delta_0) \leq_\lambda (\Delta'_2, \Delta'_1, \Delta'_0)$ if
  \begin{itemize}
    \item $|\Delta_2| \geq |\Delta'_2|$ and $|\Delta_0| \leq |\Delta'_0|$, or 
    \item $|\Delta_j| = |\Delta'_j|$ for all $j$ and 
    \begin{itemize}
      \item[$\circ$] $\langle \mu, \Delta_2 \rangle 
      := \sum_{\alpha \in \Delta_2}\langle \mu, \alpha \rangle \geq \langle \mu, \Delta'_2 \rangle$, 
      where $\mu \in \Z^2$ is such that $\mu_1 \ll \mu_2 < 0$, 
      \item[$\circ$] $\langle \lambda, \Delta_1 \rangle \geq \langle \lambda, \Delta'_1 \rangle$, and 
      \item[$\circ$] $\langle \nu, \Delta_0 \rangle \geq \langle \nu, \Delta'_0 \rangle$, 
      where $\nu \in \Z^2$ is such that $\mu_1 \ll \mu_2 < 0$. 
    \end{itemize}  
  \end{itemize}  
  \item We use the involution 
  \[
    \iota : \st_{3,n} \to \st_{3,n} : (\Delta_2, \Delta_1, \Delta_0) \mapsto (\Delta_0^t, \Delta_1^t, \Delta_2^t) , 
  \]
  which takes partial ordering $\leq_\lambda$ to partial ordering $\leq_{-(\lambda_2,\lambda_1)}$, 
  and $T^k$ to $T^{2n-k}$. 
  \end{enumerate}
\end{dfn}

``join''

\begin{thm}
\label{thm:upperTriangularity}
  Assume that the product of two classes $[\Delta_2, \Delta_1, \Delta_0]$ and $[\Delta'_2, \Delta'_1, \Delta'_0] \neq 0$ 
  is a non-zero element of the cohomology ring $\H^\star(H^n(\P^2))$ of the Hilbert scheme of points in the projective plane. 
  Then $(\Delta_2, \Delta_1, \Delta_0) \leq_\lambda \iota(\Delta'_2, \Delta'_1, \Delta'_0)$ for all general enough weights $\lambda$, 
  with equality holding only if  $(\Delta_2, \Delta_1, \Delta_0) = \iota(\Delta'_2, \Delta'_1, \Delta'_0)$.
\end{thm}

The case in which $|\Delta_j| = |\Delta'_{2-j}|$ for all $j$ is the one where Theorem \ref{thm:upperTriangularity} 
reveals most about the cohomology of $H^n(\P^2)$. 
In this case the relevant structures are partial orderings $\leq_\xi$ induced by weights 
$\xi = \mu, \lambda, \nu$ on the individual sets 
\[
  \st_m := \bigl\{ \text{standard sets } \Delta \subseteq \N^2 : |\Delta| = m \bigr\} ,
\]
in which $\Delta \leq_\xi \Delta'$ if $\langle \xi, \Delta \rangle \geq \langle \xi, \Delta' \rangle$. 
Remember the \defn{natural partial ordering}, or \defn{dominance partial ordering} on $\st_m$ \cite[p.7]{Macdonald}, 
in which $\Delta \leq \Delta'$ if the two equivalent conditions are satisfied, 
\begin{itemize}
  \item for all $i$, the sum of the sizes of the lowest $i$ rows of $\Delta$ is at most as large as the 
  sum of the sizes of the lowest $i$ rows of $\Delta'$, and 
  \item for all $j$, the sum of the sizes of the leftmost $j$ columns of $\Delta$ is at least as large as the 
  sum of the sizes of the leftmost $j$ columns of $\Delta'$. 
\end{itemize}
We will show in the Appendix that partial orderings $\leq_\xi$, for $\xi = \mu, \lambda, \nu$, 
are \defn{refinements} of $\leq$ in the sense that $\Delta \leq \Delta'$ implies $\Delta \leq_\xi \Delta'$. 
Here are a few examples to illustrate how far refinement goes. 

\begin{ex}
\label{ex:posets}
  \begin{enumerate}[(i)]
    \item For $m \leq 5$, the natural partial ordering on $\st_m$ is known to be a total ordering. 
    The partial orderings $\leq_\xi$ and $\leq$ coincide. 
    \item The natural partial ordering on $\st_6$ is known to have incomparable elements. 
    They are also incomparable in partial orderings $\leq_\xi$.  
    \item Figure \ref{fig:refineStSeven} shows the Hasse diagram of the poset $(\st_7,\leq)$, 
    arrows pointing from smaller to larger elements. 
    The three weight vectors serve as tie-breakers for elements incomparable under the natural partial ordering. 
    However, they do so in three different ways. 
    The figure also shows the additional arrows in the respective Hasse diagrams of $(\st_7,\leq_\xi)$, 
    drawn in {\color{red} red for $\xi = \mu$}, 
    {\color{blue} blue for $\xi = \lambda$}, and 
    {\color{ForestGreen} green for $\xi = \nu$}. 
    \item Figure \ref{fig:refineStEight} shows one half of the Hasse diagram of the poset $(\st_8,\leq)$, 
    the other half arising by symmetry and transposition. 
    The three weight vectors are tie-breakers for all standard sets but the two at the far right. 
    Moreover, two choices of $\lambda$ lead to two different poset structures: 
    The dashed and dotted blue arrows, respectively, show the Hasse diagram in the cases $\lambda_1 + \lambda_2 < 0$ and 
    $\lambda_1 + \lambda_2 > 0$. 
  \end{enumerate}
\end{ex}

\begin{center}
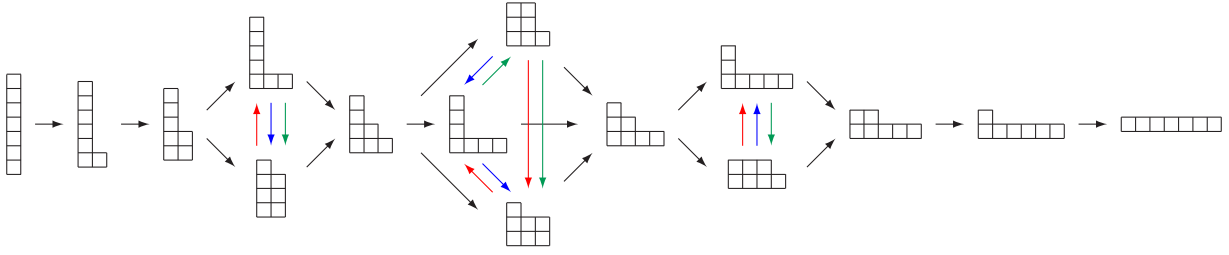
\begin{figure}[ht]
  \unitlength0.19mm
  \begin{picture}(850,170)
  \put(0,50){
    \multiput(0,0)(10,0){2}{\line(0,1){70}}
    \multiput(0,0)(0,10){8}{\line(1,0){10}}
  }
  \put(20,85){\vector(1,0){20}}
  \put(50,55){
    \multiput(0,0)(10,0){2}{\line(0,1){60}}
    \put(20,0){\line(0,1){10}}
    \multiput(0,0)(0,10){2}{\line(1,0){20}}
    \multiput(0,20)(0,10){5}{\line(1,0){10}}
  }
  \put(80,85){\vector(1,0){20}}
  \put(110,60){
    \multiput(0,0)(10,0){2}{\line(0,1){50}}
    \put(20,0){\line(0,1){20}}
    \multiput(0,0)(0,10){3}{\line(1,0){20}}
    \multiput(0,30)(0,10){3}{\line(1,0){10}}
  }
  \put(140,95){\vector(1,1){20}}
  \put(140,75){\vector(1,-1){20}}
  \put(170,110){
    \multiput(0,0)(10,0){2}{\line(0,1){50}}
    \multiput(20,0)(10,0){2}{\line(0,1){10}}
    \multiput(0,0)(0,10){2}{\line(1,0){30}}
    \multiput(0,20)(0,10){4}{\line(1,0){10}}
  }
  \put(175,20){
    \multiput(0,0)(10,0){2}{\line(0,1){40}}
    \put(20,0){\line(0,1){30}}
    \multiput(0,0)(0,10){4}{\line(1,0){20}}
    \multiput(0,40)(0,10){1}{\line(1,0){10}}
  }
  \put(210,115){\vector(1,-1){20}}
  \put(210,55){\vector(1,1){20}}
  \put(240,65){
    \multiput(0,0)(10,0){2}{\line(0,1){40}}
    \put(20,0){\line(0,1){20}}
    \put(30,0){\line(0,1){10}}
    \multiput(0,0)(0,10){2}{\line(1,0){30}}
    \multiput(0,20)(0,10){1}{\line(1,0){20}}
    \multiput(0,30)(0,10){2}{\line(1,0){10}}
  }
  \put(290,105){\vector(1,1){40}}
  \put(280,85){\vector(1,0){20}}
  \put(290,65){\vector(1,-1){40}}
  \put(350,140){
    \multiput(0,0)(10,0){3}{\line(0,1){30}}
    \put(30,0){\line(0,1){10}}
    \multiput(0,0)(0,10){2}{\line(1,0){30}}
    \multiput(0,20)(0,10){2}{\line(1,0){20}}
  }
  \put(310,65){
    \multiput(0,0)(10,0){2}{\line(0,1){40}}
    \multiput(20,0)(10,0){3}{\line(0,1){10}}
    \multiput(0,0)(0,10){2}{\line(1,0){40}}
    \multiput(0,20)(0,10){3}{\line(1,0){10}}
  }
  \put(350,0){
    \multiput(0,0)(10,0){2}{\line(0,1){30}}
    \multiput(20,0)(10,0){2}{\line(0,1){20}}
    \multiput(0,0)(0,10){3}{\line(1,0){30}}
    \multiput(0,30)(0,10){1}{\line(1,0){10}}
  }
  \put(390,125){\vector(1,-1){20}}
  \put(360,85){\vector(1,0){40}}
  \put(390,45){\vector(1,1){20}}
  \put(420,70){
    \multiput(0,0)(10,0){2}{\line(0,1){30}}
    \put(20,0){\line(0,1){20}}
    \multiput(30,0)(10,0){2}{\line(0,1){10}}
    \multiput(0,0)(0,10){2}{\line(1,0){40}}
    \multiput(0,20)(0,10){1}{\line(1,0){20}}
    \multiput(0,30)(0,10){1}{\line(1,0){10}}
  }
  \put(470,95){\vector(1,1){20}}
  \put(470,75){\vector(1,-1){20}}
  \put(500,110){
    \multiput(0,0)(10,0){2}{\line(0,1){30}}
    \multiput(20,0)(10,0){4}{\line(0,1){10}}
    \multiput(0,0)(0,10){2}{\line(1,0){50}}
    \multiput(0,20)(0,10){2}{\line(1,0){10}}
  }
  \put(505,40){
    \multiput(0,0)(10,0){4}{\line(0,1){20}}
    \put(40,0){\line(0,1){10}}
    \multiput(0,0)(0,10){2}{\line(1,0){40}}
    \multiput(0,20)(0,10){1}{\line(1,0){30}}
  }
  \put(560,115){\vector(1,-1){20}}
  \put(560,55){\vector(1,1){20}}
  \put(590,75){
    \multiput(0,0)(10,0){3}{\line(0,1){20}}
    \multiput(30,0)(10,0){3}{\line(0,1){10}}
    \multiput(0,0)(0,10){2}{\line(1,0){50}}
    \multiput(0,20)(0,10){1}{\line(1,0){20}}
  }
  \put(650,85){\vector(1,0){20}}
  \put(680,75){
    \multiput(0,0)(10,0){2}{\line(0,1){20}}
    \multiput(20,0)(10,0){5}{\line(0,1){10}}
    \multiput(0,0)(0,10){2}{\line(1,0){60}}
    \multiput(0,20)(0,10){1}{\line(1,0){10}}
  }
  \put(750,85){\vector(1,0){20}}
  \put(780,80){
    \multiput(0,0)(10,0){8}{\line(0,1){10}}
    \multiput(0,0)(0,10){2}{\line(1,0){70}}
  }
  \color{red}
  \put(175,70){\vector(0,1){30}}
  \put(365,130){\vector(0,-1){90}}
  \put(340,37.5){\vector(-1,1){20}}
  \put(515,70){\vector(0,1){30}}
  \color{blue}
  \put(185,100){\vector(0,-1){30}}
  \put(333,57.5){\vector(1,-1){20}}
  \put(340,132.5){\vector(-1,-1){20}}
  \put(525,70){\vector(0,1){30}}
  \color{ForestGreen}
  \put(195,100){\vector(0,-1){30}}
  \put(375,130){\vector(0,-1){90}}
  \put(333,112.5){\vector(1,1){20}}
  \put(535,100){\vector(0,-1){30}}
  \end{picture}
\caption{The poset $(\st_7,\leq)$ and total orderings induced by weights}
\label{fig:refineStSeven}
\end{figure}
\end{center}

\begin{center}
\begin{figure}[ht]
  \unitlength0.19mm
  \begin{picture}(500,170)
  \put(0,50){
    \multiput(0,0)(10,0){2}{\line(0,1){80}}
    \multiput(0,0)(0,10){9}{\line(1,0){10}}
  }
  \put(20,90){\vector(1,0){20}}
  \put(50,55){
    \multiput(0,0)(10,0){2}{\line(0,1){70}}
    \put(20,0){\line(0,1){10}}
    \multiput(0,0)(0,10){2}{\line(1,0){20}}
    \multiput(0,20)(0,10){6}{\line(1,0){10}}
  }
  \put(80,90){\vector(1,0){20}}
  \put(110,60){
    \multiput(0,0)(10,0){2}{\line(0,1){60}}
    \put(20,0){\line(0,1){20}}
    \multiput(0,0)(0,10){3}{\line(1,0){20}}
    \multiput(0,30)(0,10){4}{\line(1,0){10}}
  }
  \put(140,100){\vector(1,1){20}}
  \put(140,80){\vector(1,-1){20}}
  \put(170,110){
    \multiput(0,0)(10,0){2}{\line(0,1){50}}
    \multiput(20,0)(10,0){2}{\line(0,1){10}}
    \multiput(0,0)(0,10){2}{\line(1,0){30}}
    \multiput(0,20)(0,10){4}{\line(1,0){10}}
  }
  \put(175,10){
    \multiput(0,0)(10,0){2}{\line(0,1){50}}
    \multiput(20,0)(10,0){1}{\line(0,1){30}}
    \multiput(0,0)(0,10){4}{\line(1,0){20}}
    \multiput(0,40)(0,10){2}{\line(1,0){10}}
  }
  \put(210,60){\vector(1,2){20}}
  \put(210,135){\vector(1,0){20}}
  \put(210,35){\vector(1,0){20}}
  \put(240,110){
    \multiput(0,0)(10,0){2}{\line(0,1){50}}
    \put(20,0){\line(0,1){20}}
    \put(30,0){\line(0,1){10}}
    \multiput(0,0)(0,10){2}{\line(1,0){30}}
    \multiput(0,20)(0,10){1}{\line(1,0){20}}
    \multiput(0,30)(0,10){3}{\line(1,0){10}}
  }
  \put(245,15){
    \multiput(0,0)(10,0){3}{\line(0,1){40}}
    \multiput(0,0)(0,10){5}{\line(1,0){20}}
  }
  \put(280,135){\vector(1,0){80}}
  \put(280,35){\vector(1,0){20}}
  \put(310,15){
    \multiput(0,0)(10,0){2}{\line(0,1){40}}
    \put(20,0){\line(0,1){30}}
    \put(30,0){\line(0,1){10}}
    \multiput(0,0)(0,10){2}{\line(1,0){30}}
    \multiput(0,20)(0,10){2}{\line(1,0){20}}
    \multiput(0,40)(0,10){1}{\line(1,0){10}}
  }
  \put(350,35){\vector(1,0){20}}
  \put(380,15){
    \multiput(0,0)(10,0){2}{\line(0,1){40}}
    \multiput(20,0)(10,0){2}{\line(0,1){20}}
    \multiput(0,0)(0,10){3}{\line(1,0){30}}
    \multiput(0,30)(0,10){2}{\line(1,0){10}}
  }
  \put(375,110){
    \multiput(0,0)(10,0){2}{\line(0,1){50}}
    \multiput(20,0)(10,0){3}{\line(0,1){10}}
    \multiput(0,0)(0,10){2}{\line(1,0){40}}
    \multiput(0,20)(0,10){4}{\line(1,0){10}}
  }
  \put(425,135){\vector(1,0){20}}
  \put(422.5,35){\vector(1,0){25}}
  \put(422.5,55){\vector(1,2){25}}
  \put(460,115){
    \multiput(0,0)(10,0){2}{\line(0,1){40}}
    \multiput(20,0)(10,0){1}{\line(0,1){20}}
    \multiput(30,0)(10,0){2}{\line(0,1){10}}
    \multiput(0,0)(0,10){2}{\line(1,0){40}}
    \multiput(0,20)(0,10){1}{\line(1,0){20}}
    \multiput(0,30)(0,10){2}{\line(1,0){10}}
  }
  \put(465,20){
    \multiput(0,0)(10,0){3}{\line(0,1){30}}
    \multiput(30,0)(10,0){1}{\line(0,1){20}}
    \multiput(0,0)(0,10){3}{\line(1,0){30}}
    \multiput(0,30)(0,10){1}{\line(1,0){20}}
  }
  \color{red}
  \put(175,70){\vector(0,1){30}}
  \put(245,100){\vector(0,-1){30}}
  \put(395,70){\vector(0,1){30}}
  \color{blue}
  \put(185,100){\vector(0,-1){30}}
  \put(255,70){\vector(0,1){30}}
  \multiput(275,100)(4,-4){7}{\line(1,-1){3}}
  \put(302,73){\vector(1,-1){3}}
  \multiput(330,70)(4,4){7}{\line(1,1){3}}
  \put(357,97){\vector(1,1){3}}
  \multiput(373,100)(-1.5,-1.5){19}{\circle*{1}}
  \put(345,72){\vector(-1,-1){2}}
  \multiput(405,100)(0,-5){5}{\line(0,-1){4}}
  \put(405,73){\vector(0,-1){3}}
  \color{ForestGreen}
  \put(195,100){\vector(0,-1){30}}
  \put(265,70){\vector(0,1){30}}
  \put(386,100){\vector(-1,-1){30}}
  \end{picture}
\caption{The poset $(\st_8,\leq)$ and partial orderings induced by weights}
\label{fig:refineStEight}
\end{figure}
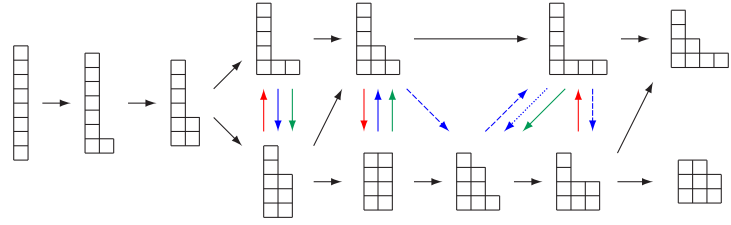
\end{center}

The upshot of these examples --- and may more which shall not be presented here --- is the following: 
The larger $m$, the farther is $\leq$ from being a total ordering. 
Weights $\mu$ and $\nu$ repair this shortcoming to a certain extent, making some but not all incomparable pairs comparable. 
For weight $\lambda$, the situation is even better: 
Though some pairs remain incomparable under $\leq_\lambda$, 
several $\lambda$ exist for which partial orderings $\leq_\lambda$ 
are different from each other. 
Theorem \ref{thm:upperTriangularity} therefore allows to derive that 
$[\Delta_2, \Delta_1, \Delta_0] \cdot [\Delta'_2, \Delta'_1, \Delta'_0] = 0$ in many more cases 
than its analogue would if we replaced $\leq_\lambda$ by the natural partial ordering. 

\begin{cor}
  For $k=0,\ldots,2n$, let $T^k$ be the basis of $\H^k(H^n(\P^2))$
  given by all $(\Delta_2, \Delta_1, \Delta_0) \in \st_{3,n}$ such that 
  $n + h(\Delta_2) - w(\Delta_0) = k$. 
  For each numbering $(t_1,\ldots,t_d)$ of $T^k$ such that whenever $t_i <_\lambda t_j$, then $i < j$, 
  the matrix of the Poincar\'e pairing 
  \[
    \H^k(H^n(\P^2)) \times \H^{2n - k}(H^n(\P^2)) \to \Z ,
  \]
  with respect to bases $(t_1,\ldots,t_d)$ on the first factor and $(\iota(t_1),\ldots,\iota(t_d))$ on the second factor 
  is upper triangular with 1s on the diagonal. 
\end{cor}

\begin{ex}
  Figure \ref{fig:fourFour} shows the Hasse diagram of the basis $T^4$ of $\H^4(H^4(\P^2))$. 
  The matrix of the Poincar\'e pairing with respect to that basis and its transpose is a block matrix
  \[
    \left[\begin{array}{c}
      t_1 \\
      \vdots \\
      t_8 \\ \hline
      t_9 \\
      \vdots \\
      t_{13} 
    \end{array}\right]
    \cdot
    \left[\begin{array}{ccc|ccc}
      \iota(t_1) &
      \ldots &
      \iota(t_8) &
      \iota(t_9) &
      \ldots &
      \iota(t_{13})
    \end{array}\right]
    = 
    \left[\begin{array}{c|c}
      \begin{array}{ccc}
        1 & & \star \\
        & \ddots & \\
        0 & & 1
      \end{array} 
      & 0 \\ \hline
      0 & 
      \begin{array}{ccc}
        1 & & \star \\
        & \ddots & \\
        0 & & 1
      \end{array} 
    \end{array}\right] .
  \]
\end{ex}

\begin{center}
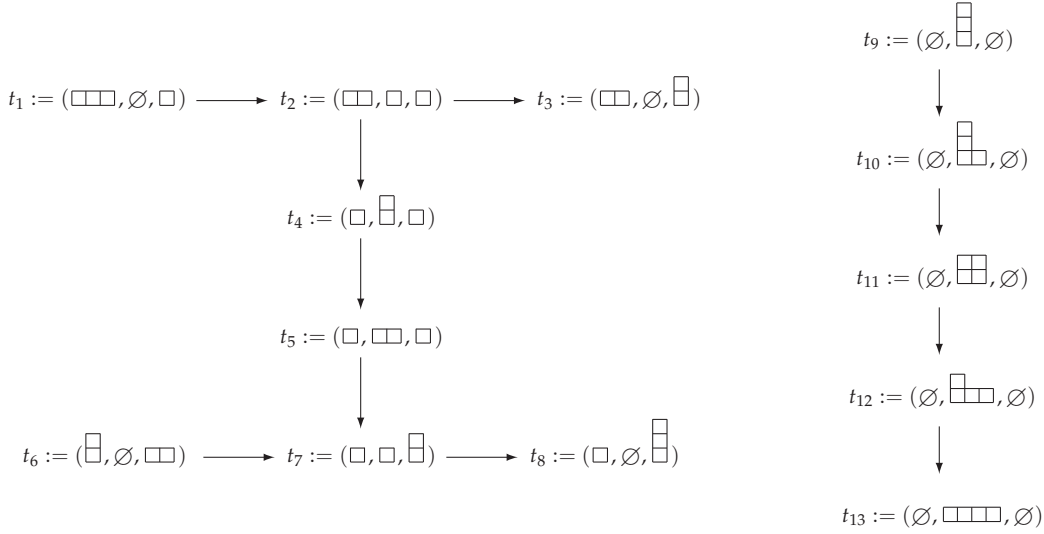
\begin{figure}[ht]
  \unitlength0.3166mm
  \begin{picture}(434,218)
    \put(0,175){\footnotesize$t_1 := (\pthree,\emptyset,\pone)$}
    \multiput(79,177)(3,-150){2}{\vector(1,0){30}}
    \put(114,175){\footnotesize$t_2 := (\ptwo,\pone,\pone)$}
    \multiput(187,177)(-3,-150){2}{\vector(1,0){30}}
    \multiput(148,169)(0,-50){3}{\vector(0,-1){30}}
    \put(222,175){\footnotesize$t_3 := (\ptwo,\emptyset,\poneone)$}
    \put(117,125){\footnotesize$t_4 := (\pone,\poneone,\pone)$}
    \put(114,75){\footnotesize$t_5 := (\pone,\ptwo,\pone)$}
    \put(6,25){\footnotesize$t_6 := (\poneone,\emptyset,\ptwo)$}
    \put(117,25){\footnotesize$t_7 := (\pone,\pone,\poneone)$}
    \put(219,25){\footnotesize$t_8 := (\pone,\emptyset,\poneoneone)$}
    \put(359,200){\footnotesize$t_9 := (\emptyset,\poneoneone,\emptyset)$}
    \multiput(391,190)(0,-50){4}{\vector(0,-1){20}}
    \put(356,150){\footnotesize$t_{10} := (\emptyset,\ptwooneone,\emptyset)$}
    \put(356,100){\footnotesize$t_{11} := (\emptyset,\ptwotwo,\emptyset)$}
    \put(353,50){\footnotesize$t_{12} := (\emptyset,\pthreeone,\emptyset)$}
    \put(350,0){\footnotesize$t_{13} := (\emptyset,\pfour,\emptyset)$}
  \end{picture}
\caption{The poset basis of $\H^4(H^4(\P^2))$}
\label{fig:fourFour}
\end{figure}
\end{center}

\subsection*{Outline of the paper}

Section \ref{sec:esRevisited} gives a short account of Ellingsrud and Str\o mme's BB cells in schemes 
$H^n(\A^2)$, $H^{n,\lin}(\A^2)$ and $H^{n,\punc}(\A^2)$. 
The emphasis of this section is on the identification of ideals in $S$ and triples of ideals in $S'$. 
We will conclude that section with the remark that Ellingsrud and Str\o mme's cellular decomposition of $H^n(\P^2)$
is not a BB decomposition. 
Section \ref{sec:bbTheory} provides the token from Bia\l ynicki-Birula theory necessary for the proof of Theorem \ref{thm:upperTriangularity}: 
Consider a complete complex variety $X$ equipped with an ample line bundle and a torus action with isolated fixed points, 
and the BB cells $X_v^\circ$ of points floating into $v$ from above and $X^w_\circ$ of points floating into $v$ from below. 
We will give a necessary condition in terms of the line bundle for $X_v^\circ$ to meet $X^w_\circ$. 
Hence also a necessary condition for the respective closures to meet. 
The main source of inspiration to that section was \cite{allenSimplicialComplexesBB}. 
Section \ref{sec:poincare} then applies the findings from Bia\l ynicki-Birula theory, 
plus some elementary observations from \cite{comboDuality}, to the scheme $H^n(\P^2)$, thus proving Theorem \ref{thm:upperTriangularity}. 
The Appendix contains a few clarifications about the partial orderings $\leq_\mu$, $\leq_\lambda$, $\leq_\nu$ 
as discussed in Example \ref{ex:posets}. 
Moreover, we will specify generic monomial ideals of given weights, 
a notion we use in the proof of Theorem \ref{thm:upperTriangularity}.

\subsection*{Acknowledgements}

I wish to thank Anthony Iarrobino and Bernd Sturmfels for giving me the opportunity to present preliminary versions of this work 
at Northwestern University at Boston and Max Planck Institut f\"ur Mathematik in Bonn, respectively. 
Bernd noted the analogy of my results to a statement in toric varieties, as presented in Example \ref{ex:bernd}. 
I then rewrote my paper according to Bernd's suggestions with the help of Allen Knutson, whom I thank for many very valuable remarks. 


\section{Ellingsrud and Str\o mme's cellular decomposition}
\label{sec:esRevisited}

We use the cellular decomposition of $\P^2 = \Proj(S)$ 
into $\A^2 = \P^2 \setminus \V(x_2)$, $\A^1 = \V(x_2) \setminus \{(1:0:0)\}$ and $\A^0 = \{(1:0:0)\}$. 
An ideal $I \in H^n(\P^2)$ can uniquely be written as $I = I_2 \cap I_1 \cap I_0$, where $\supp(I_j) \subseteq \A^j$. 
The individual ideals $I_j$ have constant Hilbert functions $n_j$ whose sum equals $n$. 
Hence the decomposition of $H^n(\P^2)$ into strata 
$H^{n_2}(\A^2) \times H^{n_1,\lin}(\A^2) \times H^{n_0,\punc}(\A^2)$ as presented in the Introduction. 
We identify ideals $I = I_2 \cap I_1 \cap I_0$ in $S$ with triples $(I_2,I_1,I_0)$ of ideals in $S'$ by de-homogenizing, 
\[
  I(x_0,x_1,x_1) \mapsto \bigl( I_2(y_1,y_2,1), I_1(y_1,1,y_2), I_0(1,y_1,y_2) \bigr) . 
\] 

The two-dimensional torus $\TT$ of diagonal matrices in $SL(3)$ acts on the polynomial ring $S$ by scaling the variables: 
If $\tt = (t_0,t_1,t_2) \in \TT$ then $\tt . x^\alpha := t_0^{\alpha_0}t_1^{\alpha_1}t_2^{\alpha_2}x^\alpha$. 
This translates into an action on $\P^2$ where $\tt.(a_0:a_1:a_2) = (t_0^{\alpha_0}a_0:t_1^{\alpha_1}a_1:t_2^{\alpha_2}a_2)$. 
The fixed points of this action are $(1:0:0)$, $(0:1:0)$ and $(0:0:1)$. 
The $\TT$-action on $S$ also induces an action on $H^n(\P^2)$ by $\tt . I = \langle t.f : f \in I \rangle$. 
This action respects the decomposition into strata $H^{n_2}(\A^2) \times H^{n_1,\lin}(\A^2) \times H^{n_0,\punc}(\A^2)$. 
In particular, the fixed points under this action are ideals of shape $I = M_{\Delta_2} \cap M_{\Delta_1} \cap M_{\Delta_0}$, 
where each $M_{\Delta_j} \subseteq S$ is a monomial ideal supported in $\A^j$. 
Equivalently, the fixed points are triples $(M_{\Delta_2}, M_{\Delta_1}, M_{\Delta_0})$ of monomial ideals $M_{\Delta_j} \subseteq S'$. 

The datum of a \defn{weight vector} $w := (w_0,w_1,w_2) \in \Z^3$ such that $w_0 + w_1 + w_2 = 0$
is equivalent to the datum of an embedding
\[
  T := \CC^\star \hookrightarrow \TT : t \mapsto (t^{w_0},t^{w_1},t^{w_2}) .
\]
We obtain induced actions of the one-dimensional torus $T$ on the geometric objects $\P^2$ and $H^n(\P^2)$. 
The action of $T$ on $H^n(\P^2)$ also respects its decomposition into strata $H^{n_2}(\A^2) \times H^{n_1,\lin}(\A^2) \times H^{n_0,\punc}(\A^2)$. 
If the weight vector is general enough, 
then the $T$-fixed points on $H^n(\P^2)$ are still triples $(M_{\Delta_2}, M_{\Delta_1}, M_{\Delta_0})$ of monomial ideals in $S'$. 
The BB strata of that action are therefore affine cells contained in $H^{n_2}(\A^2)$, $H^{n_1,\lin}(\A^2)$ and $H^{n_0,\punc}(\A^2)$, respectively. 
For determining a basis of the cohomology module of $H^n(\P^2)$, it suffices to find (possibly different) weight vectors 
such that we know the BB strata of the respective actions in $H^{n_2}(\A^2)$, $H^{n_1,\lin}(\A^2)$ and $H^{n_0,\punc}(\A^2)$ explicitly enough. 

\begin{pro}[cf. \cite{esBetti, esCells, huibregtseEllingsrud, hilbertBurchMatrices}]
\label{pro:BBCells}
  \begin{enumerate}[(i)]
  \item General weight vectors $w, w' \in \Z^3$ exist such that 
  \begin{itemize}
    \item $w_0 < w_1 < w_2$ and $w_0 + w_1 + w_2 = 0$, 
    \item the same holds for the primed vector, 
    \item $w_0-w_2 \ll w_1-w_2 < 0$, more precisely, $w_0-w_2 < n(w_1-w_2)$, 
    \item $w_0-w_1 < 0 < w_2-w_1$, 
    \item $w'_0-w'_1 < 0 < w'_2-w'_1$, and 
    \item $0 < w'_1-w'_0 \ll w'_2-w'_0$, more precisely, $w'_1-w'_0 < n(w'_2-w'_0)$. 
  \end{itemize}
  \item The weight $w$ defines an action $T \acts H^n(\P^2)$ with isolated fixed points 
  $(M_{\Delta_2},M_{\Delta_1},M_{\Delta_0})$ such that 
  the BB sink of points floating into $(M_{\Delta_2},M_{\Delta_1},\langle 1 \rangle)$ from above 
  is the subscheme $H^{\Delta_2}_\lex(\A^2) \times H^{\Delta_1,\lin}_\lex(\A^2) \times \emptyset$. 
  \item The weight $w'$ defines an action $T \acts H^n(\P^2)$ with isolated fixed points 
  $(M_{\Delta_2},M_{\Delta_1},M_{\Delta_0})$ such that 
  the BB sink of points floating into $(\langle 1 \rangle, M_{\Delta_1},M_{\Delta_0})$ from above 
  is the subscheme $\emptyset \times H^{\Delta_1,\lin}_\lex(\A^2) \times H^{\Delta_0,\punc}_\lex(\A^2)$. 
  \end{enumerate}
\end{pro}

\begin{proof}
  (i) is elementary. 
  
  (ii) Under the identification of ideals in $S$ and triples of ideals in $S'$, 
  the action $T \acts H^n(\P^2)$ translates into three actions of $T$ on $H^{n_2}(\A^2)$, $H^{n_1}(\A^2)$ and $H^{n_0}(\A^2)$, 
  respectively, induced by actions 
  \[
    \begin{array}{cc}
      T \acts S' : 
      & t.y^\alpha = t^{\langle \alpha, (w_0 - w_2, w_1 - w_2) \rangle}y^\alpha , \\
      T \acts S' : 
      & t.y^\alpha = t^{\langle \alpha, (w_0 - w_1, w_2 - w_1) \rangle}y^\alpha , \\
      T \acts S' : 
      & t.y^\alpha = t^{\langle \alpha, (w_1 - w_0, w_2 - w_0) \rangle}y^\alpha ,
    \end{array}
  \]
  respectively. 
  Here it is important to note that the weight vector of the second action is such that $w_0 - w_1 < 0$ and $w_2 - w_1 > 0$, 
  to the effect that the BB sink of points floating into a fixed point $M_{\Delta_1}$ from above is contained in $H^{n_1,\lin}(\A^2)$; 
  and the weight vector of the third action is such that $w_0 - w_1 > 0$ and $w_2 - w_1 > 0$, 
  to the effect that the BB sink of points floating into a fixed point $M_{\Delta_0}$ from above is contained in $H^{n_1,\punc}(\A^2)$. 
  Both facts have been used in \cite{esBetti, esCells}, and are proved on a schematic level in \cite{bbAvecLaurent}. 
  However, for the time being, we will only be using the first and the second of the three actions; 
  we will return to the third action in Section \ref{sec:poincare}. 
  
  The first action sends each polynomial $f = \sum_{\alpha \in \N^2}a_\alpha y^\alpha \in S'$ to 
  $t.f = \sum_{\alpha \in \N^2}a_\alpha t^{\langle (w_0 - w_2, w_1 - w_2),\alpha \rangle}y^\alpha$. 
  Let $a_\beta y^\beta$ be the \defn{weight-initial term of $f$}, i.e., 
  the term for which the product $\langle (w_0 - w_2, w_1 - w_2),\beta \rangle$ is minimal. 
  Our choice of $(w_0 - w_2, w_1 - w_2)$ immediately shows that the weight-initial term of $f$ is its lex-initial term. 
  Under the $T$ action on $H^{n_2}(\A^2)$, an ideal $I_2$ is sent to 
  \[
    t.I_2 = \left\langle \frac{t.f}{t^{\langle (w_0 - w_2, w_1 - w_2),\beta \rangle}} 
    = \sum_{\alpha \in \N^2}a_\alpha t^{\langle (w_0 - w_2, w_1 - w_2),\alpha - \beta \rangle}y^\alpha : f \in I \right\rangle
  \]
  In the limit as $t \to 0$, all summands $a_\alpha t^{\langle (w_0 - w_2, w_1 - w_2),\alpha - \beta \rangle}$ for which 
  $\langle (w_0 - w_2, w_1 - w_2),\alpha \rangle < \langle (w_0 - w_2, w_1 - w_2),\beta \rangle$ get killed. 
  Since these terms are the lex-trailing terms, 
  the limit is the lexicographic Gr\"obner deformation. 
  The BB sink of points $I_2 \in H^{n_2}(\A^2)$ floating into fixed point $M_{\Delta_2}$ therefore equals $H^{\Delta_2}_\lex(\A^2)$. 
  
  The second action sends each polynomial $f$ to 
  $t.f = \sum_{\alpha \in \N^2}a_\alpha t^{\langle (w_0 - w_1, w_2 - w_1),\alpha \rangle}y^\alpha$. 
  The weight-initial term of $f$ is now the summand $a_\beta y^\beta$
  for which the product $\langle (w_0 - w_1, w_2 - w_1),\beta \rangle$ is minimal. 
  Given any $f \in S'$, its weight-initial term is clearly not its lex-initial term. 
  We claim that if $f$ is an element of the reduced lexicographic Gr\"obner basis of an ideal $I_1 \in H^{n_1,\lin}(\A^2)$, then it is. 
  Once this claim is proved, the same arguments as before show that
  the BB sink of points $I_1 \in H^{n_1,\lin}(\A^2)$ floating into $M_{\Delta_1}$ equals $H^{\Delta_1,\lin}_\lex(\A^2)$. 
  
  Let us prove our claim in the most direct way, without reference to \cite{esCells, huibregtseEllingsrud, hilbertBurchMatrices}. 
  It suffices to show that if $\IN(f) = y^\alpha$, then $f$ contains no term $a_\beta y^\beta$ such that $\beta_2 < \alpha_2$. 
  We establish this by applying a series of deformations to $I_1$. 
  Each will be a limit $\lim_{t \to 0} t.I_1$ under a torus action induced from $T \acts S'$ with a weight vector $u \in \Z^2$. 
  
  The first action has weight vector $u := (-1,0)$. 
  Since $I_1$ is supported on the $y_1$-axis, 
  the deformation $\lim_{t \to 0} t.I_1$ is an ideal defining a point in $H^{n_1,\punc}(\A^2)$ invariant under this torus action. 
  However, the weight $u$ was chosen such that each $f = y^\alpha + \sum_{\beta \in \Delta_1, \beta < \alpha} a_\beta y^\beta$
  appearing in the reduced lexicographic Gr\"obner basis of $I_1$ deforms to 
  $\lim_{t \to 0}t.f = y^\alpha + \sum_{\beta_1 = \alpha_1, \beta_2 < \alpha_2} a_\beta y^\beta$. 
  The limits of the polynomials $f$ form the reduced lexicographic Gr\"obner basis of the limiting ideal. 
  The ideal spanned by them can only be supported at the origin if $a_\beta = 0$
  for all $\beta$ such that $\beta_1 = \alpha_1$ and $\beta_2 < \alpha_2$. 
  
  For defining the next action, we turn the weight vector $(-1,0)$ clockwise and rescale it 
  so as to obtain $u \in \Z^2$ pointing slightly upward into the third quadrant of the plane. 
  We stop turning when we first find an element from the reduced lexicographic Gr\"obner basis of $I_1$ 
  having initial term $y^\alpha$ and a trailing term $a_\beta y^\beta$ such that $\langle u, \alpha - \beta \rangle = 0$. 
  The deformation $\lim_{t \to 0} t.I_1$ then defines a point in $H^{n_1,\punc}(\A^2)$ invariant under this torus action. 
  The reduced lexicographic Gr\"obner basis of the limiting ideal is formed by polynomials $\lim_{t \to 0} t.f$, 
  where $f$ runs through the reduced lexicographic Gr\"obner basis of $I_1$. 
  Once more we see that the ideal spanned by these polynomials can only be supported at the origin if $a_\beta = 0$
  for all $\beta$ such that $\langle u, \alpha - \beta \rangle = 0$. 
  
  We keep turning $u$ clockwise and deforming $I_1$, 
  thus showing that more and more coefficients $a_\beta$ vanish. 
  The arguments remain valid as long as $u$ stays in the interior of third quadrant. 
  This finishes the proof of the claim. 
  
  (iii) The action $T \acts H^n(\P^2)$ translates into actions on the three types of Hilbert schemes induced by 
  \[
    \begin{array}{cc}
      T \acts S' : 
      & t.y^\alpha = t^{\langle \alpha, (w'_0 - w'_2, w'_1 - w'_2) \rangle}y^\alpha , \\
      T \acts S' : 
      & t.y^\alpha = t^{\langle \alpha, (w'_0 - w'_1, w'_2 - w'_1) \rangle}y^\alpha , \\
      T \acts S' : 
      & t.y^\alpha = t^{\langle \alpha, (w'_1 - w'_0, w'_2 - w'_0) \rangle}y^\alpha ,
    \end{array}
  \]
  respectively. 
  Only the last two actions are relevant in (ii). 
  The second action here is identical to the second action in (ii). 
  The BB sink of points $I_1 \in H^{n_1,\lin}(\A^2)$ floating into $M_{\Delta_1}$ therefore equals $H^{\Delta_1,\lin}_\lex(\A^2)$. 
  
  For proving that the BB sink of points $I_0 \in H^{n_0,\punc}(\A^2)$ 
  floating into $M_{\Delta_0}$ equals $H^{\Delta_0,\punc}_\lex(\A^2)$, 
  we show that if $f$ is an element of the reduced lexicographic Gr\"obner basis of an ideal $I_0 \in H^{n_0,\punc}(\A^2)$, 
  then its initial term with respect to the weight $(w'_1 - w'_0, w'_2 - w'_0)$ is at the same time its lex-initial term. 
  It suffices to show that if $\IN(f) = y^\alpha$, then $f$ contains no term $a_\beta y^\beta$ such that $\beta_2 < \alpha_2$
  or $\beta_2 = \alpha_2$ and $\beta_1 < \alpha_1$.
  
  The non-existence of $a_\beta y^\beta$ such that $\beta_2 < \alpha_2$ follows from the fact that the ideal $I_0$ 
  also defines a point in $H^{n_0,\lin}(\A^2)$. 
  For showing the non-existence of $a_\beta y^\beta$ such that $\beta_2 = \alpha_2$ and $\beta_1 < \alpha_1$, 
  we consider the torus action on $H^{n_0,\punc}(\A^2)$ induced from $T \acts S'$ with weight vector $u := (0,1)$. 
  Since $I_0$ is supported at the origin, 
  the deformation $\lim_{t \to 0} t.I_0$ is an ideal defining a point in $H^{n_0,\punc}(\A^2)$ invariant under this torus action.
  The polynomials $\lim_{t \to 0}t f$, for all $f$ from the reduced lexicographic Gr\"obner basis of $I_0$, 
  form the reduced lexicographic Gr\"obner basis of the limiting ideal. 
  The ideal spanned by them can only be supported at the origin if $a_\beta = 0$
  for all $\beta$ such that $\beta_2 = \alpha_2$ and $\beta_1 < \alpha_1$. 
\end{proof}

It's not possible to impose the conditions from Proposition \ref{pro:BBCells} (i) for one and the same weight vector $w = w'$. 
This is why Ellingsrud and Str\o mme's cellular decomposition is not a BB decomposition. 
The next two sections will be dedicated to a somewhat more sophisticated application of Bia\l ynicki-Birula theory to our setting. 



\section{Some complements on Bia\l ynicki-Birula theory}
\label{sec:bbTheory}

The following proposition was communicated to the author by Allen Knutson. 

\begin{pro}
\label{pro:allenBB}
  Let $X$ be a complete complex variety with an ample line bundle $\L$ 
  and $T \acts X$ an action with isolated fixed points. 
  We denote by 
  \begin{equation*}
  \begin{split}
    X_v^\circ & := \{ x \in X : \lim_{t \to 0} t .x = v \} , \\
    X^w_\circ & := \{ x \in X : \lim_{t \to \infty} t .x = w \}
  \end{split}
  \end{equation*}
  the BB cells of points flowing into $v$ and $w$ from above and below, respectively, 
  and by $\Phi(v)$ the $T$-weight of $\L |_v$. 
  If $X_v^\circ \cap X^w_\circ$ is nonempty and $v \neq w$, then $\Phi(v) < \Phi(w)$.
\end{pro}

\begin{proof}
  This statement and its proof are the same spirit as \cite[Proposition 2.1]{allenSimplicialComplexesBB}. 
  A general point $a \in X$ defines a morphism $f: T \to X: t \mapsto t.a$. 
  Since the $T$-action is not assumed to be faithful, the map $f$ is generically $k : 1$ for some $k \geq 1$. 
  The projective line $\P^1$ is a compactification of $T$; 
  if $a \in X_v^\circ \cap X^w_\circ$, then the morphism $f$ extends to $g: \P^1 \to X$, 
  sending $0$ and $\infty$ to $v$ and $w$, respectively.
  Completeness of $X$ implies that the extension of $f$ is unique. 
  
  The pullback of $\L$ along $g$ is a line bundle $\L'$ on $\P^1$ 
  whose only poles and zeros are found at 0 and $\infty$. 
  Their respective degrees are $\Phi(v)$ and $\Phi(w)$, and 
  \[
    \deg(\L') = \Phi(w) - \Phi(v) .
  \]
  The pushforward of $\L'$ along $g$ 
  shows that the degree of $\L'$ equals $k$ times the degree of the restriction of $\L$ to the image of $g$. 
  The line bundle $\L$ being ample, it has a positive degree, as does its restriction to the image of $g$. 
  Thus the inequality $\Phi(w) - \Phi(v) > 0$ follows. 
\end{proof}

\begin{lmm}
\label{lmm:closure}
  In the setting of Proposition \ref{pro:allenBB}, let $a \in X_v := \overline{X_v^\circ}$ and $p := \lim_{t \to 0} t. a$. 
  If $v \neq p$, then $\Phi(v) < \Phi(p)$. 
\end{lmm}

\begin{proof}
  The line bundle $\L$ pulls back from $X$ to the complete variety $X_v$. 
  In the special case $\L = \O(1)$, the statement of the lemma is part of \cite[Corollary 2.1]{allenSimplicialComplexesBB}. 
  The proof of the cited corollary is based on \cite[Lemma 2.1]{allenSimplicialComplexesBB}, 
  which states a number of facts about the line bundle $\O(1)$ on $X_v$. 
  These statements are proved by pulling back the line bundle to $\P^1$ 
  by a morphism $g : \P^1 \to \P^1$ as in the proof of Proposition \ref{pro:allenBB}. 
  The proof of the lemma is based on the classification of equivariant line bundles on $\P^1$, 
  which is given by the degree and the torus weight on the fiber over 0. 
  This classification, however, doesn't care if the line bundle on $\P^1$ 
  is the pullback of $\O(1)$ or of a more general $\L$. 
  The proof of \cite[Corollary 2.1]{allenSimplicialComplexesBB} 
  therefore works for an arbitrary line bundle $\L$ on $X_v$. 
\end{proof}

\begin{thm}
\label{thm:intersectionOfClosures}
  In the setting of Proposition \ref{pro:allenBB}, $X_v \cap X^w \neq \emptyset$ only if $v = w$ or $\Phi(v) < \Phi(w)$. 
\end{thm}

\begin{proof}
  Consider a point $a \in X_v \cap X^w$ and its limits 
  $p := \lim_{t \to 0} t.a$ and $q := \lim_{t \to \infty} t.a$. 
  Lemma \ref{lmm:closure}, its analogue for limits as $t \to \infty$, and Proposition \ref{pro:allenBB} show that 
  \[
    \Phi(v) \leq \Phi(p) \leq \Phi(q) \leq \Phi(w) , 
  \]
  with strict inequality if any two of the fixed points don't agree.   
\end{proof}

\begin{cor}
\label{cor:poincare}
  In the situation of Theorem \ref{thm:intersectionOfClosures}, 
  pick all $X_v$ of a given dimension $m-k$, where $m = \dim(X)$ thus obtaining a basis of $H^k(X)$, 
  and pick all $X^w$ of dimension $k$, thus obtaining a basis of $H^{m-k}(X)$. 
  Then the matrix of the Poincar\'e pairing $H^k(X) \times H^{m-k}(X) \to \Z$ with respect to the two bases, 
  ordered in the weight-increasing way, is upper triangular with 1s on the diagonal. 
\end{cor}

\begin{proof}
  Completeness of $X$ implies that the Bia\l ynicki-Birula strata $X_v^\circ$ 
  with $v$ running through the set of fixed points, form a cellular decomposition of $X$ \cite[Theorem 4.4]{bialynickiBirula}, 
  as do the strata $X^w_\circ$. 
  The cohomology classes of their closures are therefore a basis of the module $H^\star(X)$. 
  The two bases of $H^k(X)$ and $H^{m-k}(X)$, respectively, are obtained by picking elements from the appropriate graded pieces. 
  Upper triangularity of the matrix follows from Theorem \ref{thm:intersectionOfClosures}. 
  Since the Poincar\'e pairing is perfect, the diagonal entries of the matrix are 1. 
\end{proof}

\begin{rmk}
  Another statement from Bia\l ynicki-Birula's article allows to understand the 1s on the diagonal in a very direct way. 
  The tangent spaces $T_v^+(X) := T_v(X_v^\circ)$ and $T_v^-(X) := T_w(X^v_\circ)$ are called the 
  \defn{positive} and \defn{negative weight spaces}, respectively. 
  In the setting of Proposition \ref{pro:allenBB}, $\Phi(v) = \Phi(w)$ if, and only if, $v = w$, 
  in which case the tangent space of $X$ at $v$ decomposes into 
  \[
    T_v(X) = T_v^+(X) \oplus T_v^-(X) 
  \]
  by \cite[Theorems 4.1 and 4.4]{bialynickiBirula}. 
  In other words, the closures $X_v$ and $X^v$ meet transversally in vertex $v$. 
  Since the proof of \cite[Theorem 4.2]{bialynickiBirula} also shows that $v$ is the only point in $X_v \cap X^v$, 
  the matrix of the Poincar\'e pairing has a 1 in row $X_v$ and column $X^v$.
\end{rmk}

The author thanks Bernd Sturmfels for communicating the following example
and generously providing a folded napkin akin to the polytope from Figure \ref{fig:polytope}. 
The reason why this example is discussed in so much detail here 
is its role as a model case for the Hilbert scheme of points in the projective plane, 
to be discussed in the next section. 

\begin{center}
\begin{figure}[ht]
  \begin{picture}(170,200)
    \put(40,10){\line(-1,3){30}}
    \put(40,10){\line(4,1){40}}
    \put(80,20){\line(1,1){40}}
    \put(120,60){\line(0,1){60}}
    \put(10,100){\line(1,2){43.4}}
    \put(120,120){\line(-1,1){66.8}}
    \put(35,2){\footnotesize $v_0$}
    \put(80,12){\footnotesize $v_1$}
    \put(124,55){\footnotesize $v_2$}
    \put(-2,99){\footnotesize $v_3$}
    \put(124,120){\footnotesize $v_4$}
    \put(50,192){\footnotesize $v_5$}
    \multiput(80,20)(-4,0){7}{\line(-1,0){2}}
    \multiput(120,60)(-4,0){7}{\line(-1,0){2}}
    \multiput(120,120)(-4,0){7}{\line(-1,0){2}}
    \multiput(10,100)(4,0){7}{\line(1,0){2}}
    \put(160,100){\vector(0,-1){40}}
    \put(90,160){\footnotesize $P \subseteq M \otimes_\Z \R$}
    \put(170,80){\footnotesize $\lambda \in M^\star$}
    \color{red}
    \thicklines
    \put(10,100){\qbezier(15,0)(15,9)(6.71,13.42)
    \line(1,2){6.8}}
    \put(26,110){\footnotesize $(X_P)_{v_3}^\circ$}
    \put(78,70){\footnotesize $(X_P)_{v_2}^\circ$}
    \put(40,10){\qbezier(-2.37,7.115)(5,8.5)(7.275,1.82)
    \line(-1,3){2.4}}
    \put(40,10){\line(4,1){7.45}}
    \put(80,20){\qbezier(-15,0)(-15,11)(-5.5,14.5)
      \qbezier(-5.5,14.5)(2.1,17.3)(10.61,10.61)
      \line(1,1){10.8}
      }
    \put(120,120){\qbezier(-15,0)(-15,6)(-10.61,10.61)
    \line(-1,1){10.8}}
    \put(120,60){\qbezier(-15,0)(-15,15)(0,15)
    \line(0,1){15}}
    \color{blue}
    \put(10,100){\qbezier(15,0)(14,-11)(4.74,-14.23)
    \line(1,-3){4.8}}
    \put(26,85){\footnotesize $(X_P)^{v_3}_\circ$}
    \put(78,45){\footnotesize $(X_P)^{v_2}_\circ$}
    \put(80,20){\qbezier(-15,0)(-15,-3.64)(-14.55,-3.64)
    \line(-4,-1){14.9}}
    \put(120,60){\qbezier(-15,0)(-15,-6)(-10.61,-10.61)
    \line(-1,-1){10.8}}
    \put(120,120){\qbezier(-15,0)(-15,-15)(0,-15)
    \line(0,-1){15}}
    \put(53.2,186.8){\qbezier(-6.71,-13.42)(3,-18)(10.61,-10.61)
    \line(-1,-2){6.8}}
    \put(53.2,186.8){\line(1,-1){10.8}}
    \put(40,10){\circle*{3}}
    \color{red}
    \put(53.3,186.7){\circle*{3}}
  \end{picture}
\caption{A polytope $P$; a linear form $\lambda$ defining a direction of flow; 
fixed points $v_i$ ordered according to their weight; 
points in {\color{red}$(X_P)_{v_i}^\circ$} and {\color{blue}$(X_P)^{v_i}_\circ$}
are below and above the dashed lines, respectively;
${\color{red}(X_P)_{v_i}} \cap {\color{blue}(X_P)^{v_j}} \neq \emptyset$ only if $i \leq j$.}
\label{fig:polytope}
\end{figure}
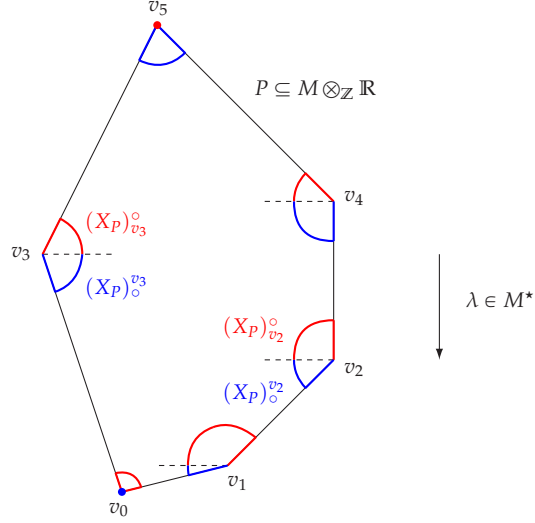
\end{center}

\begin{ex}
\label{ex:bernd}
  Let $M \cong \Z^d$ be a lattice. 
  A full dimensional smooth lattice polytope $P \subseteq M \otimes_\Z \R$, as pictured in Figure \ref{fig:polytope}, 
  defines a smooth projective variety $X_P$. 
  Topologically, $X_P$ is the closure of the image of $\TT \hookrightarrow \P^{s-1}$ given by $\tt \mapsto (\tt^\alpha)_{\alpha \in P \cap M}$. 
  In particular, each vertex $v \in P$ defines a point $v \in X_P$. 
  Schematically, $X_P$ is embedded into projective space $\P^{s-1}$ whose homogeneous coordinate ring 
  is generated by indeterminates $x_\alpha$, one for each $\alpha \in P \cap M$. 
  An open affine cover of $X_P$ is given by schemes $U_v := \Spec S_v$, one for each vertex $v \in P$, where 
  \[
    S_v := \CC\left[\frac{x_\alpha}{x_v} : \alpha \in P \cap M\right] / I_v ,
  \]
  the binomial ideal $I_v$ implementing all $\Z$-linear relations between lattice vectors $\alpha - v$, for $\alpha \in M$. 
  
  The $d$-torus $\TT = \Spec \CC[M]$ naturally acts on each coordinate ring, 
  $\TT \acts S_v : \tt.\frac{x_\alpha}{x_v} = \tt^{v-\alpha} \frac{x_\alpha}{x_v}$. 
  The induced actions on patches $U_v$ are compatible with each other, 
  hence a $\TT$-action on the entire variety $X_P$. 
  The orbit-cone correspondence \cite[Theorem 3.2.6]{coxLittleSchenck} implies that 
  vertices $v \in P$ correspond to maximal cones $\sigma_v$ in the dual fan $\Sigma_P$ of $P$, 
  which in turn correspond to $\TT$-fixed points in $X_P$.  
  We therefore denote by $v \in X_P$ the fixed point corresponding to vertex $v \in P$. 
  
  We pull back the ample line bundle $\O(1)$ on $\P^{s-1}$ to an ample line bundle $\L$ on $X_P$. 
  Let's compute, for each fixed point $v \in X_P$, the weight of the $\TT$-representation $\L|_v$. 
  We denote by $E_v$ the set of $\alpha \in P \cap M$ that are reached first when walking away from $v$ on edges of $P$. 
  Smoothness of $P$ says that for each $v$, the set $\{ \alpha - v : \alpha \in E_v \}$ is a $\Z$-basis of $M$. 
  Therefore, $U_v = \Spec S'_v \cong \A^d$, where 
  \[
    S'_v = \CC\left[\frac{x_\alpha}{x_v} : \alpha \in E_v\right] .
  \]
  Locally around $v$, the sheaf $\L$ is isomorphic to the restriction of $\O_{\P^d}(1)$ to $U_v$, 
  whose module of global sections is $\Gamma(U_v,\O(1)) = x_v \cdot S'_v$. 
  The $\TT$-action on that module is given by $\tt.(x_v \cdot\frac{x_\alpha}{x_v}) = \tt^{-\alpha} (x_v \cdot \frac{x_\alpha}{x_v})$.  
  This formula holds for all $\alpha \in E_v$ and also for $\alpha = v$. 
  In particular, the $\TT$-action on secions of $\L$ of shape ``$x_v$ times a constant'' is given by multiplication by $\tt^{-v}$.
  The stalk $\L|_v \cong \CC$ contains precisely those elements, so the $\TT$-action on $\L|_v$ is multiplication by $\tt^{-v}$. 
  In other words, $\L|_v$ has $\TT$-weight $-v$. 
  
  Consider a linear form $\lambda: M \to \Z$ separating the vertices of $P$, 
  i.e., $\langle \lambda,v \rangle \neq \langle \lambda,w \rangle$ for vertices $v \neq w$. 
  The element $\lambda \in M^\star$ corresponds to an embedding $T \hookrightarrow \TT$, 
  hence an action $T \acts X_P$. 
  The general choice of $\lambda$ implies that $T$-fixed points of $X_P$ correspond to vertices of $P$. 
  Restricting the $\TT$-action to the subtorus $T$ amounts to replacing $\tt^v$ by $t^{\langle \lambda,v \rangle}$. 
  The $T$-weight on $\L|_v$ is therefore $\Phi(v) = -\langle \lambda,v \rangle$. 
  Thus the assumptions of Proposition \ref{pro:allenBB} are satisfied. 
  
  The Bia\l ynicki-Birula cells $(X_P)_v^\circ$ and $(X_P)^v_\circ$ have explicit characterizations. 
  Both types of cells are contained in the open neighborhood $U_v$ of $v$. 
  The torus action on closed points, i.e., maximal ideals in $S'_v$, is given by 
  \[
    t.\left\langle \frac{x_\alpha}{x_v} - a_\alpha : \alpha \in E_v \right\rangle 
    = \left\langle t^{\langle \lambda, v - \alpha \rangle} \frac{x_\alpha}{x_v} - a_\alpha : \alpha \in E_v \right\rangle
    = \left\langle \frac{x_\alpha}{x_v} - t^{\langle \lambda, \alpha-v \rangle} a_\alpha : \alpha \in E_v \right\rangle .
  \]
  The limit as $t \to 0$ of this ideal exists in $S'_v$ if, and only if, $a_\alpha = 0$ 
  for all $\alpha \in E_v$ such that $\langle \lambda, \alpha-v \rangle < 0$. Thus 
  \[
    (X_P)_v^\circ = \Spec S'_v / \left\langle \frac{x_\alpha}{x_v} : \langle\lambda, \alpha-v \rangle < 0 \right\rangle ,
  \]
  and similarly for $(X_P)^v_\circ$, whose ideal is defined by the property $\langle\lambda, \alpha-v \rangle > 0$. 
  Accordingly, Figure \ref{fig:polytope} shows points in $(X_P)_v^\circ$ lying above the dashed line through $v$ 
  and points in $(X_P)^v_\circ$ below them. 
  The figure illustrates the statement of Theorem \ref{thm:intersectionOfClosures} and Corollary \ref{cor:poincare}: 
  \begin{itemize}
    \item $(X_P)_{v_i}$ only meets $(X_P)^{v_j}$ if $v_i = v_j$ or $\Phi(v_i) < \Phi(v_j)$, and 
    \item $(X_P)_{v_i}$ and $(X_P)^{v_i}$ meet transversally in $v_i$, 
    the tangent space decomposing into vectors above and below the dashed line through $v_i$. 
  \end{itemize}
\end{ex}


\section{Torus weights of triples of monomial ideals}
\label{sec:poincare}

Now we apply the findings of the previous section to the scheme $H^n(\P^2)$. 
For constructing a line bundle on that scheme, 
we embed it into a suitable projective space and pull back $\O(1)$. 
Remember that statements (ii) and (iii) from Proposition \ref{pro:BBCells} only characterized the BB sinks 
of two special types of torus fixed points, $(M_{\Delta_2},M_{\Delta_1},\emptyset)$ and $(\emptyset,M_{\Delta_1},M_{\Delta_0})$, 
respectively, in terms of lexicographic Gr\"obner basins. 
The technique for doing so was based on a study of weights $w$ and $w'$. 
We will be investigating more general fixed points $(M_{\Delta_2},M_{\Delta_1},M_{\Delta_0})$, 
using more general weights $w''$. 
We will always choose $w''$
general enough so that the only fixed points of the induced action $T \acts H^n(\P^2)$ are monomial ideals in $S$. 

\begin{lmm}
\label{lmm:closedImmersion}
  Consider the closed immersion 
  \[
    \iota: H^n(\P^2) \to G^n_{S_n} \to \P^N , 
  \]
  where 
  \begin{itemize}
    \item the first map is the closed immersion from into the Grassmannian of rank $n$ quotients of 
    the $d$-th graded piece of the polynomial ring $S$, for some $d \geq n$, and 
    \item the second map is the Pl\"ucker embedding. 
  \end{itemize}
  Let $\L$ be the pullback of the line bundle $\O(1)$ on $\P^N$ by that immersion.
  Let $w'' \in \Z^3$ be a weight such that 
  \begin{itemize}
    \item $w''_0 + w''_1 + w''_3 = 0$, 
    \item $w''_0 < w''_1 < w''_2$, and 
    \item the only fixed points of the induced action $T \acts H^n(\P^2)$ are monomial ideals in $S$, 
    identified with triples of monomial ideals in $S'$.  
  \end{itemize} 
  We denote by $\Phi_{w''}(\Delta_2,\Delta_1,\Delta_0) \in \Z$ the induced $T$-weight of $\L |_{(M_{\Delta_2},M_{\Delta_1},M_{\Delta_0})}$. 
  Then 
  \[
    \Phi_{w''}(\Delta_2,\Delta_1,\Delta_0) = - \langle {w''}, \Delta_{<d} \rangle 
    - \langle {w''}, \iota_2(\Delta_2) \rangle - \langle {w''}, \iota_1(\Delta_1) \rangle - \langle {w''}, \iota_0(\Delta_0) \rangle , 
  \]
  where $\Delta_{<d}$ is the simplex in $\N^3$ of all elements of degrees $<d$ and 
  \[
    \iota_j : \N^2 \to \N^3 : (\alpha_1,\alpha_2) \mapsto 
    \begin{cases}
      (\alpha_1,\alpha_2,d-\alpha_1-\alpha_2) & \text{if } j = 2 , \\
      (\alpha_1,d-\alpha_1-\alpha_2,\alpha_2) & \text{if } j = 1 , \\
      (d-\alpha_1-\alpha_2,\alpha_1,\alpha_2) & \text{if } j = 0 . 
    \end{cases}
  \]
\end{lmm}

Figure \ref{lmm:closedImmersion} illustrates how immersions $\iota_j$ identify triples of monomial ideals in $S'$ with 
special types of monomial ideals in $S$. 
It also shows the range from which weights $w''$ will be taken, 
along with the particular weights $w$ and $w'$ from Proposition \ref{pro:BBCells}. 
All vectors are drawn such that their tails sit in the barycenter of the simplex $\left\{\alpha \in \N^3 : |\alpha| = d\right\}$. 

\begin{center}
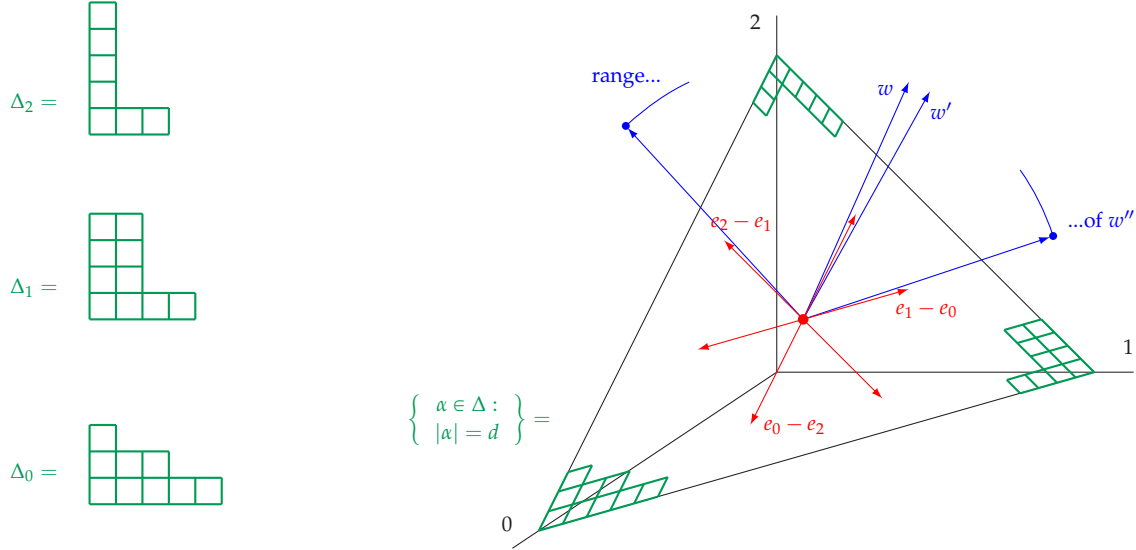
\begin{figure}[ht]
  \begin{picture}(425,195)
    \put(290,60){\line(-3,-2){100}}
    \put(290,60){\line(1,0){135}}
    \put(290,60){\line(0,1){135}}
    \put(290,180){\line(1,-1){120}}
    \put(290,180){\line(-90,-180){90}}
    \put(200,0){\line(210,60){210}}
    \put(186,0){\footnotesize $0$}
    \put(421,67){\footnotesize $1$}
    \put(280,190){\footnotesize $2$}
    \color{ForestGreen}
    \thicklines
    \put(30,150){
      \put(-30,10){\footnotesize $\Delta_2 = $}
      \multiput(0,0)(10,0){2}{\line(0,1){50}}
      \multiput(20,0)(10,0){2}{\line(0,1){10}}
      \multiput(0,0)(0,10){2}{\line(1,0){30}}
      \multiput(0,20)(0,10){4}{\line(1,0){10}}
    }
    \put(30,80){
      \put(-30,10){\footnotesize $\Delta_1 = $}
      \multiput(0,0)(10,0){3}{\line(0,1){40}}
      \multiput(0,0)(0,10){2}{\line(1,0){40}}
      \multiput(30,0)(10,0){2}{\line(0,1){10}}
      \multiput(0,20)(0,10){3}{\line(1,0){20}}
    }
    \put(30,10){
      \put(-30,10){\footnotesize $\Delta_0 = $}
      \multiput(0,0)(10,0){2}{\line(0,1){30}}
      \multiput(20,0)(10,0){2}{\line(0,1){20}}
      \multiput(40,0)(10,0){2}{\line(0,1){10}}
      \multiput(0,0)(0,10){2}{\line(1,0){50}}
      \put(0,20){\line(1,0){30}}
      \put(0,30){\line(1,0){10}}
    }
    \put(290,180){
      \multiput(0,0)(-3,-6){2}{\line(1,-1){25}}
      \multiput(-6,-12)(-3,-6){2}{\line(1,-1){5}}
      \multiput(0,0)(5,-5){2}{\line(-1,-2){9}}
      \multiput(10,-10)(5,-5){4}{\line(-1,-2){3}}
    }
    \put(410,60){
      \multiput(0,0)(-7,-2){3}{\line(-1,1){20}}
      \multiput(0,0)(-5,5){2}{\line(-21,-6){28}}
      \multiput(-21,-6)(-7,-2){2}{\line(-1,1){5}}
      \multiput(-10,10)(-5,5){3}{\line(-21,-6){14}}
    }
    \put(200,0){
      \multiput(36,10.29)(9,2.57){2}{\line(1,2){3.7}}
      \multiput(18,5.14)(9,2.57){2}{\line(1,2){7.4}}
      \multiput(0,0)(9,2.57){2}{\line(1,2){11.1}}
      \multiput(0,0)(3.7,7.4){2}{\line(7,2){45}}
      \put(7.4,14.8){\line(9,2.57){27}}
      \put(11.1,22.2){\line(9,2.57){9}}
    }
    \put(150,40){\footnotesize $\left\{\begin{array}{c} \alpha \in \Delta : \\ |\alpha| = d \end{array}\right\} = $}
    \thinlines
    \color{blue}
    \put(300,80){\vector(80,180){40}}
    \put(328,165){\footnotesize $w$}
    \put(300,80){\vector(100,180){48}}
    \put(348,156){\footnotesize $w'$}
    \put(300,80){\vector(-110,120){66.5}}
    \put(300,80){\vector(210,70){93.5}}
    \put(232.9,153.2){\circle*{3}}
    \put(232.9,153.2){\qbezier(0,0)(12,11)(24,16.6)}
    \put(394.5,111.5){\circle*{3}}
    \put(394.5,111.5){\qbezier(0,0)(-4.67,14)(-12.5,25)}
    \put(220,170){\footnotesize range...}
    \put(400,115){\footnotesize ...of $w''$}
    \color{red}
    \put(300,80){\circle*{4}}
    \put(300,80){\vector(-90,-180){20}}
    \put(300,80){\vector(90,180){20}}
    \put(300,80){\vector(-120,120){30}}
    \put(300,80){\vector(120,-120){30}}
    \put(300,80){\vector(210,60){40}}
    \put(300,80){\vector(-210,-60){40}}
    \put(285,38){\footnotesize $e_0 - e_2$}
    \put(335,82){\footnotesize $e_1 - e_0$}
    \put(265,115){\footnotesize $e_2 - e_1$}
  \end{picture}
\caption{Identification of triples of monomial ideals in $S'$ and monomial ideals in $S$, 
and the weights $w$ and $w'$}
\label{fig:triplesStandardSets}
\end{figure}
\end{center}

\begin{proof}[Proof of Lemma \ref{lmm:closedImmersion}]
  The first of the two immersions is described in 
  \cite[\S VI.1]{bayerThesis}, \cite[Proposition C.30]{IarrobinoKanev} and \cite[Theorem 4.4]{multigradedHilbertSchemes}
  for arbitrary Grothendieck Hilbert schemes. 
  The scheme $H^n(\P^2)$ is a special case in which the immersion is described thusly: 
  Take a point $a$ in $H^n(\P^2)$, say, an $A$-valued one, where $A$ is any $\CC$-algebra. 
  This point is is a homogeneous ideal $I \subseteq S \otimes_\CC A$ such that for large $d$, 
  the quotient $(S \otimes_\CC A)_d / I_d$ is a locally free $A$-module of rank $n$. 
  After replacing $I$ by its saturation, we may assume that the quotient is locally free of rank $n$ for all $d \geq n$. 
  The last inequality is a consequence of the Gotzmann number of the constant Hilbert polynomial $n$ being equal to $n$. 
  After passing to a Zariski-open subset of $\Spec A$, 
  we may assume that $(S \otimes_\CC A)_d / I_d$ is free of rank $n$. 
  The epimorphism of $A$-modules $(S \otimes_\CC A)_d \to (S \otimes_\CC A)_d / I_d$ is therefore represented by a matrix
  $C$ with entries in $A$ with $n$ rows and ${n + 2 \choose 2}$ columns. 
  The columns represent the images of the monomials $y^\alpha$ in $(S \otimes_\CC A)_d / I_d$, 
  written in terms of a basis of that quotient. 
  
  After passing to a smaller Zariski-open subset of $\Spec A$, 
  we may assume a basis of $(S \otimes_\CC A)_d / I_d$ to be given by all $y^\beta$ of total degree $d$ lying in the standard set $\Delta$ 
  of a monomial ideal $M_\Delta \subseteq S$. 
  The matrix $C$ therefore contains an identity block indexed by columns $y^\beta$, for all $\beta \in \Delta$ of total degree $d$. 
  For each $y^\alpha \in M$ of total degree $d_0$, 
  there exist unique $c^\alpha_\beta \in A$ such that 
  \[
    y^\alpha = \sum_{\beta \in \Delta, |\beta| = d_0} a^\alpha_\beta y^\beta \in (S \otimes_\CC A)_d / I_d , 
  \]
  or equivalently, 
  \[
    f_\alpha := y^\alpha - \sum_{\beta \in \Delta, |\beta| = d_0} a^\alpha_\beta y^\beta \in I_{d_0} . 
  \]
  Column $y^\alpha$ of matrix $C$ therefore has entries $a^\alpha_\beta \in A$ if $\alpha \in M_\Delta$, 
  and $\delta^\alpha_\beta$ otherwise. 
  Upon reordering the columns of $C$, we obtain $C = \left( \begin{array}{cc} E & A \end{array} \right)$, 
  where $E$ is the identity matrix and $A$ is the matrix of coefficients $a^\alpha_\beta$ of polynomials $f_\alpha$. 
  The matrix $C$ defines a point in the Grassmannian $G^n_{S_n}$. 
  The cited results from \cite{bayerThesis, IarrobinoKanev, multigradedHilbertSchemes} 
  say that assigning to point $a \in H^n(\P^2)$ that point in the Grassmannian defines a closed immersion  
  \[
    H^n(\P^2) \to G^n_{S_n} ,
  \]
  and that we may play the same game in any degree $d \geq n$: 
  For each $y^\alpha \in M$ of total degree $d$, 
  find unique $b^\alpha_\beta \in A$ such that 
  \[
    f_\alpha := y^\alpha - \sum_{\beta \in \Delta, |\beta| = d} b^\alpha_\beta y^\beta \in I_d ; 
  \]
  write their coefficients into a matrix $D = \left( \begin{array}{cc} E & B \end{array} \right)$; 
  get a map 
  \[
    H^n(\P^2) \to G^n_{S_d} .
  \]
  This map is a closed immersion factoring through the first closed immersion. 

  The Pl\"ucker embedding then sends the point represented by matrix $D$ to the point 
  represented by a one-row matrix whose entries are the maximal minors of the matrix $D$. 
  For tracing the effect of the action $T \acts H^n(\P^2)$ on that image point, 
  we first observe that 
  \[
    \frac{t.f_\alpha}{t^{\langle w, \alpha \rangle}} 
    = y^\alpha - \sum_{\beta \in \Delta, |\beta| = d} t^{\langle w, \beta-\alpha \rangle} b^\alpha_\beta y^\beta \in I_d , 
  \]
  which has the entry $b^\alpha_\beta$ of matrix $D$ replaced by $t^{\langle w, \beta-\alpha \rangle} b^\alpha_\beta$. 
  We may multiply the entire $\alpha$th row of the rescaled matrix with $t^{\langle w, \alpha \rangle}$ without changing 
  its point in the Grassmannian. 
  The $T$-action may therefore be understood to just scale column $y^\beta$ of $D$ by the factor $t^{\langle w, \beta \rangle}$. 
  After applying the Pl\"ucker embedding, the action therefore sends a point with homogeneous coordinates $(\det(C'_D))_D$ 
  to the point with homogeneous coordinates $(t^{\langle w, D \rangle}\det(C'_D))_D$, 
  where $\langle w, D \rangle := \sum_{\beta \in D}\langle w, \beta \rangle$. 
  
  Locally around the $T$-fixed point $M_\Delta \subseteq S$, projective space $\P^N$ is an affine space 
  with coordinate ring 
  \[
    S_\Delta := \CC\left[\frac{x_E}{x_\Delta} : E \subseteq \{\alpha \in \N^3 : |\alpha| = d\}, |E| = n, E \neq \Delta \right] 
  \]
  We have just shown that the action 
  \[
    T \acts S_\Delta : t.\frac{x_E}{x_\Delta} := t^{-\langle w, E \rangle + \langle w, \Delta \rangle}\frac{x_E}{x_\Delta}
  \]
  induces an action on projective space 
  equivariant with respect to the immersion $H^n(\P^2) \hookrightarrow \P^N$. 
  The same arguments as in Example \ref{ex:bernd} show that the stalk $\O(1)|_{M_\Delta}$ 
  of line bundle $\O(1)$ on $\P^N$ has $T$-weight $-\langle w, \Delta \rangle$. 
  Therefore, so does the stalk $\L|_{M_\Delta}$ of line bundle $\L := \iota^\star \O(1)$ on $H^n(\P^2)$. 
  The statement of the lemma is now just a reformulation of this formula in terms of triples of monomial ideals in $S'$. 
\end{proof}

In the proof of Theorem \ref{thm:upperTriangularity}, 
we will be using torus actions $T \acts S'$ with general weights $u,v \in \Z^2$ such that $u_1,u_2 < 0$ and $v_1,v_2 > 0$, respectively, 
and the induced BB decompositions 
\begin{equation*}
  \begin{split}
    H^n(\A^2) & = \coprod_{|\Delta| = n} H^{\Delta}_u(\A^2) , \\
    H^{n,\punc}(\A^2) & = \coprod_{|\Delta| = n} H^{\Delta,\punc}_v(\A^2) . 
  \end{split}
\end{equation*}
The respective cofactors parametrize ideals $I \subseteq S'$
floating into $M_\Delta$ from above under the action induced by $T \acts S'$ with weight vectors $u$ and $v$, respectively. 
When using weights different from the ones considered by Ellingsrud and Str\o mme, 
the structure of the BB sinks is in general hard to determine \cite{constantinescu}. 
However, it turns out that we only have to work with very specific types of monomial ideals. 

\begin{dfn}
\label{dfn:genericMonomialIdeals}
  Schemes $H^n(\A^2)$ and $H^{n,\punc}(\A^2)$ are smooth and irreducible of dimensions $2n$ and $n-1$, 
  respectively \cite{Fogarty_smoothness,briancon}. 
  The above-displayed BB decompositions therefore 
  contain unique members of maximal dimensions $2n$ and $n-1$, respectively. 
  We call the corresponding torus fixed point the \defn{generic monomial ideals} 
  in $H^n(\A^2)$ and $H^{n,\punc}(\A^2)$
  with respect to the weights $u$ and $v$, respectively.   
\end{dfn}

The shape of generic standard sets for a given weight shall be specified in the Appendix.

\begin{proof}[Proof of Theorem \ref{thm:upperTriangularity}]
  
  Under our identification of ideals $I = I_2 \cap I_1 \cap I_0$ in $S$ 
  with triples $(I_2,I_1,I_0)$ of ideals in $S'$, the BB sinks of an action $T \acts H^n(\P^2)$ 
  inherited from $T \acts S$ with a general weight $w''$ take the shape 
  \[
  \begin{array}{ccl}
    \BB(w'')_{(\Delta_2, \Delta_1, \Delta_0)}^\circ & := & 
    \left\{ \begin{array}{c} (I_2,I_1,I_0) \in H^{n_2}(\A^2) \times H^{n_1,\lin}(\A^2) \times H^{n_0,\punc}(\A^2) : \\
    \lim_{t \to 0} t._{w''} (I_2,I_1,I_0) = (M_{\Delta_2},M_{\Delta_1},M_{\Delta_0}) \end{array} \right\} \\
    & = & H^{\Delta_2}_{w''(2)}(\A^2) \times H^{\Delta_1,\lin}_\lex(\A^2) \times H^{\Delta_0,\punc}_{w''(0)}(\A^2) , \\
    \BB(w'')^{(\Delta_2, \Delta_1, \Delta_0)}_\circ & := & 
    \left\{ \begin{array}{c}  (I_2,I_1,I_0) \in H^{n_2,\punc}(\A^2) \times H^{n_1,\lin}(\A^2) \times H^{n_0}(\A^2) : \\ 
    \lim_{t \to \infty} t._{w''} (I_2,I_1,I_0) = (M_{\Delta_2},M_{\Delta_1},M_{\Delta_0}) \end{array} \right\} \\
    & \simeq & H^{\Delta_0}_{-w''(0)}(\A^2) \times H^{(\Delta'_1)^t,\lin}_\lex(\A^2) \times H^{\Delta_2,\punc}_{-w''(2)}(\A^2) . 
  \end{array}
  \]
  Here and in what follows we use the notation $w''(2) := (w''_0 - w''_2, w''_1 - w''_2)$ and analogues for indices $1$ and $0$. 
  As for the middle factor in the first formula, we might have used the weight $w''(1)$ as a subscript; 
  however, as was shown in the proof of Proposition \ref{pro:BBCells}, 
  the weight $w''(1)$ plus the fact that we consider ideals supported in $\V(y_2)$ implies that the BB sink 
  is the same thing as the lexicographic Gr\"obner basin. 
  As for the the middle factor in the second formula, we used the identity 
  \[
    H^{\Delta'_1,\lin}_{-w''(1)}(\A^2) = H^{(\Delta'_1)^t,\lin}_\lex(\A^2) 
  \]
  coming from the fact that the weight $-w''(1)$ is positive in the first component and negative in the second component. 
  Moreover, in the second of the above-displayed formul\ae\ we reversed the order of the three factors, 
  so as to get them into our usual ``plane--line--point'' shape. 
    
  Take two homology classes such that $[\Delta_2, \Delta_1, \Delta_0] \cdot [\Delta'_2, \Delta'_1, \Delta'_0] \neq 0$, 
  then $(\Delta_2, \Delta_1, \Delta_0) \cap (\Delta'_2, \Delta'_1, \Delta'_0) \neq \emptyset$. 
  Remember that the last two varieties are not closures of BB cells. 
  The idea of our proof is to embed them into closures of BB cells 
  and use those ambient BB cells as approximations of the varieties we wish to study. 
  We use a number of weights $w''$ for deriving inequalities on the standard sets $\Delta_i,\Delta'_j$. 
  Independently of the weight chosen, inclusions
  \[
    \begin{array}{cccccc}
      H^{\Delta_2}_\lex(\A^2) & \subseteq & H^{\Gamma_2}_{w''(2)}(\A^2) , & 
      H^{\Delta_0,\punc}_\lex(\A^2) & \subseteq & H^{\Gamma_0}_{w''(0)}(\A^2) , \\
      H^{\Delta'_2}_\lex(\A^2) & \subseteq & H^{\Gamma'_2}_{-w''(0)}(\A^2) , & 
      H^{\Delta'_0,\punc}_\lex(\A^2) & \subseteq & H^{\Gamma'_0}_{-w''(2)}(\A^2) , \\
    \end{array}
  \]
  hold true, where $M_{\Gamma_i}$ denotes the generic monomial ideals for the respective weights. 
  Therefore, 
  \[
    \begin{array}{ccccc}
      (\Delta_2, \Delta_1, \Delta_0) & \subseteq & 
      \overline{ H^{\Gamma_2}_{w''(2)}(\A^2) \times H^{\Delta_1,\lin}_{w''(1)}(\A^2) \times H^{\Gamma_0,\punc}_{w''(0)}(\A^2) } & = & 
      \BB(w'')_{(\Gamma_2, \Delta_1, \Gamma_0)} , \\
      (\Delta'_2, \Delta'_1, \Delta'_0) & \subseteq & 
      \overline{ H^{\Gamma'_2}_{-w''(0)}(\A^2) \times H^{\Delta'_1,\lin}_{w''(1)}(\A^2) \times H^{\Gamma'_0,\punc}_{-w''(2)}(\A^2) } & \simeq & 
      \BB(w'')^{(\Gamma'_0, (\Delta'_1)^t, \Gamma'_2)} .
    \end{array}
  \]
  Note that the second formula doesn't state equality but rather just isomorphism, 
  given by reversing the roles of variables $x_0,x_1,x_2$. 
  Since the spaces on the left-hand side intersect nontrivially, so do the spaces on the right-hand side. 
  Theorem \ref{thm:intersectionOfClosures} therefore says that 
  \[
    \Phi_{w''}(\Gamma_2,\Delta_1,\Gamma_0) \leq \Phi_{w''}(\Gamma'_0, (\Delta'_1)^t, \Gamma'_2) .
  \]
  Using the explicit formula for weights from Lemma \ref{lmm:closedImmersion}, this inequality reads 
  \begin{equation*}
    \begin{split}
      & 
      - \langle w''(2), \Gamma_2 \rangle  
      - \langle w''(1), \Delta_1 \rangle  
      - \langle w''(0), \Gamma_0 \rangle 
      - w''_2|\Gamma_2|d - w''_1|\Delta_1|d - w''_0|\Gamma_0|d \\    
      \leq & 
      - \langle w''(2), \Gamma'_0 \rangle  
      - \langle w''(1), (\Delta'_1)^t \rangle  
      - \langle w''(0), \Gamma'_2 \rangle 
      - w''_2|\Gamma'_0|d - w''_1|(\Delta'_1)^t|d - w''_0|\Gamma'_2|d .
    \end{split}
  \end{equation*}
  
  If we choose $d \gg n$, then only the last three summands on either side 
  contribute to the inequality. We are then left with inequality
  \[
    \langle w'', (|\Gamma_0| - |\Gamma'_2|, |\Delta_1| - |\Delta'_1|, |\Gamma_2| - |\Gamma'_0|) \rangle 
    \geq 0 ,
  \]
  which says that the angle between the two vectors in the plane is at most $\pi/2$. 
  At this point we use two specific weights $w''$. 
  The first weight of choice is such that $w''_2-w''_1 \gg w''_1-w''_0$. 
  Visually this is the blue vector pointing to the far left in Figure \ref{fig:triplesStandardSets}. 
  Then the previous inequality amounts to  
  \[
    |\Gamma_2| - |\Gamma'_0| = |\Delta_2| - |\Delta'_0| \geq 0 .
  \]
  The second of choice is such that $w''_1-w''_0 \gg w''_2-w''_1$. 
  Visually this is the blue vector pointing to the far right in Figure \ref{fig:triplesStandardSets}. 
  Then the above inequality amounts to  
  \[
    |\Gamma_0| - |\Gamma'_2| = |\Delta_0| - |\Delta'_2| \leq 0 .
  \]
  We thus obtain 
  \begin{itemize}
    \item $|\Delta_2| \geq |\Delta'_0|$ and $|\Delta_0| \leq |\Delta'_2|$,
  \end{itemize}
  i.e., the first of the two bulleted conditions from Definition \ref{dfn:orderings}. 
  
  Say equality holds here. Then also $|\Delta_1| = |\Delta'_1|$. 
  At this point we make explicit use of affine cells 
  $(\Delta_2, \Delta_1, \Delta_0)^\circ$ and $(\Delta'_2, \Delta'_1, \Delta'_0)^\circ$ in $H^n(\P^2)$ 
  as visualized in Figure \ref{fig:triples}. 
  As for the first cell, we use it literally as defined in the Introduction, 
  so $(\Delta_2, \Delta_1, \Delta_0)^\circ$ parametrizes triples $(I_2,I_1,I_0)$ of ideals
  \begin{itemize}
    \item[$\circ$] $I_2 \subseteq \CC[\frac{x_0}{x_2},\frac{x_1}{x_2}]$ such that $\IN_\lex(I_2) = M_{\Delta_2}$, 
    where $\frac{x_0}{x_2} > \frac{x_1}{x_2}$; 
    \item[$\circ$] $I_1 \subseteq \CC[\frac{x_0}{x_1},\frac{x_2}{x_1}]$ supported in $\V(\frac{x_2}{x_1})$ such that $\IN_\lex(I_1) = M_{\Delta_1}$, 
    where $\frac{x_0}{x_1} > \frac{x_2}{x_1}$; and 
    \item[$\circ$] $I_0 \subseteq \CC[\frac{x_1}{x_0},\frac{x_2}{x_0}]$ supported at the origin such that $\IN_\lex(I_0) = M_{\Delta_0}$, 
    where $\frac{x_1}{x_0} > \frac{x_2}{x_0}$. 
  \end{itemize}
  Hence the picture on the left-hand side of Figure \ref{fig:triples}. 
  For indicating the ``direction of flow'', the generic monomial ideals in $\CC[\frac{x_0}{x_2},\frac{x_1}{x_2}]$ and 
  $\CC[\frac{x_1}{x_0},\frac{x_2}{x_0}]$, respectively, are drawn in blue, 
  and the least generic ones (whose cells are in fact points) are drawn in blue. 
  This is in analogy to the use of red and blue in Figure \ref{fig:polytope}. 
  As for the second cell, defined by reversing the order of the variables. 
  Read in this way, $(\Delta'_2, \Delta'_1, \Delta'_0)^\circ$ parametrizes triples of ideals
  \begin{itemize}
    \item[$\circ$] $I_2 \subseteq \CC[\frac{x_0}{x_2},\frac{x_1}{x_2}]$ supported at the origin such that $\IN_\lex(I_2) = M_{\Delta_2}$, 
    where $\frac{x_1}{x_2} > \frac{x_0}{x_2}$; 
    \item[$\circ$] $I_1 \subseteq \CC[\frac{x_0}{x_1},\frac{x_2}{x_1}]$ supported in $\V(\frac{x_0}{x_1})$ such that $\IN_\lex(I_1) = M_{\Delta_1}$, 
    where $\frac{x_2}{x_1} > \frac{x_0}{x_1}$; and 
    \item[$\circ$] $I_0 \subseteq \CC[\frac{x_1}{x_0},\frac{x_2}{x_0}]$ such that $\IN_\lex(I_0) = M_{\Delta_0}$, 
    where $\frac{x_2}{x_0} > \frac{x_1}{x_0}$. 
  \end{itemize}
  Hence the picture on the left-hand side of Figure \ref{fig:triples}. 
  Note the new ``direction of flow'', which got reversed when reversing the order of the variables. 
  \begin{center}
  \begin{figure}[ht]
    \begin{picture}(450,180)
      \multiput(0,0)(250,0){2}{
        \put(80,60){\line(-3,-2){80}}
        \put(80,60){\line(0,1){120}}
        \put(80,60){\line(1,0){120}}
        \multiput(80,160)(60,0){2}{\line(-3,-2){40}}
        \multiput(80,160)(-40,-26.67){2}{\line(1,0){60}}
        \multiput(180,60)(0,60){2}{\line(-3,-2){40}}
        \multiput(180,60)(-40,-26.67){2}{\line(0,1){60}}
        \multiput(13.34,15.56)(60,0){2}{\line(0,1){60}}
        \multiput(13.34,15.56)(0,60){2}{\line(1,0){60}}
        \put(2,0){\footnotesize $0$}
        \put(195,64){\footnotesize $1$}
        \put(73,175){\footnotesize $2$}
      }
      \put(32,20){\footnotesize $\V(I_0)$}
      \put(160,35){\footnotesize $\V(I_1)$}
      \put(120,135){\footnotesize $\V(I_2)$}
      \put(300,30){\footnotesize $\V(I'_2)$}
      \put(436,85){\footnotesize $\V(I'_1)$}
      \put(336,148){\footnotesize $\V(I'_0)$}
      \put(13.34,15.56){\circle*{8}}
      \multiput(174,56)(-12,-8){2}{\circle*{4}}
      \multiput(168,52)(-15,-10){2}{\circle*{3}}
      \put(147,38){\circle*{3}}
      \put(330,160){\circle*{8}}
      \multiput(430,63)(0,20){2}{\circle*{3}}
      \multiput(430,70)(0,35){2}{\circle*{4}}
      \put(430,98){\circle*{3}}
      \multiput(174,56)(-12,-8){2}{\circle*{4}}
      \multiput(100,150)(-15,-10){2}{\circle*{2}}
      \multiput(89,155)(-12,-8){2}{\circle*{3}}
      \multiput(83,151)(-18,-12){2}{\circle*{2}}
      \multiput(110,152)(-18,-12){2}{\circle*{2}}
      \multiput(273,30)(0,25){2}{\circle*{2}}
      \multiput(280,50)(0,18){2}{\circle*{2}}
      \multiput(290,25)(0,35){2}{\circle*{3}}
      \multiput(290,35)(0,18){2}{\circle*{2}}
      \color{blue}
        \multiput(67,160)(5,0){2}{\line(-3,-2){20}}
        \multiput(67,160)(-3.33,-2.22){7}{\line(1,0){5}}
        \multiput(13.34,2.56)(0,5){2}{\line(1,0){30}}
        \multiput(13.34,2.56)(5,0){7}{\line(0,1){5}}
        \multiput(338.67,165.78)(-3.33,-2.22){2}{\line(1,0){30}}
        \multiput(338.67,165.78)(5,0){7}{\line(-3,-2){3.33}}
        \multiput(250.34,15.56)(5,0){2}{\line(0,1){30}}
        \multiput(250.34,15.56)(0,5){7}{\line(1,0){5}}
      \color{red}
        \multiput(88.67,165.78)(-3.33,-2.22){2}{\line(1,0){30}}
        \multiput(88.67,165.78)(5,0){7}{\line(-3,-2){3.33}}
        \multiput(.34,15.56)(5,0){2}{\line(0,1){30}}
        \multiput(.34,15.56)(0,5){7}{\line(1,0){5}}
        \multiput(317,160)(5,0){2}{\line(-3,-2){20}}
        \multiput(317,160)(-3.33,-2.22){7}{\line(1,0){5}}
        \multiput(263.34,2.56)(0,5){2}{\line(1,0){30}}
        \multiput(263.34,2.56)(5,0){7}{\line(0,1){5}}
    \end{picture}
  \caption{Points in $(\Delta_2, \Delta_1, \Delta_0)^\circ$ and the coordinate-reversed version of $(\Delta'_2, \Delta'_1, \Delta'_0)^\circ$, 
  and the {\color{red}most generic} and {\color{blue}least generic} monomial ideals on four factors}
  \label{fig:triples}
  \end{figure}
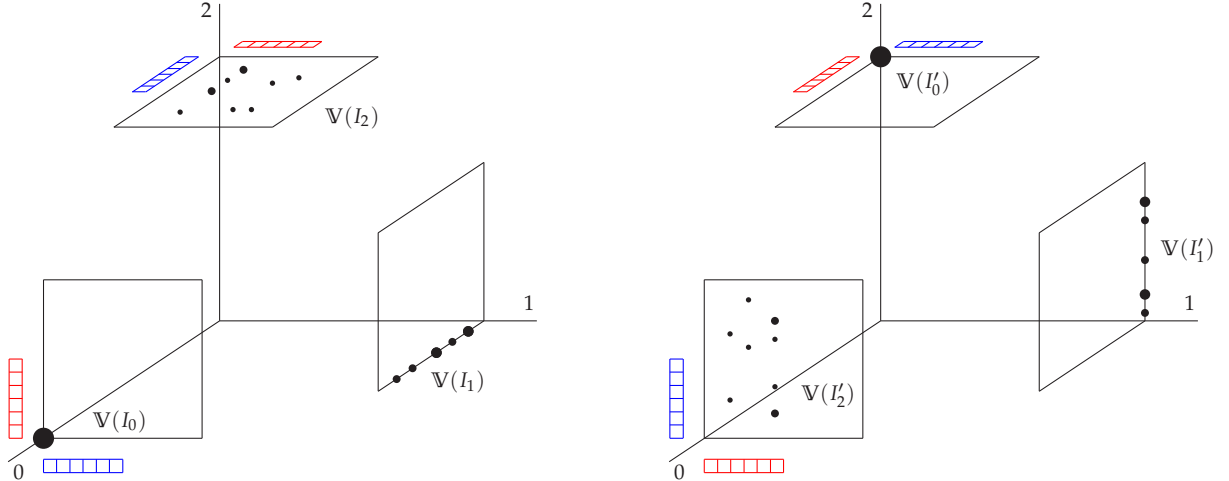
  \end{center}
  
  Passing to the closures of the two affine cells, 
  it's fairly easy to see that the assumption that $n_j := |\Delta_j| = |\Delta'_{2-j}|$ for $j = 0,1,2$, plus the fact that 
  \begin{itemize}
    \item[$\circ$] the support of $I_1$ can't leave the line $\V(x_2) \subseteq \P^2$,
    \item[$\circ$] the support of $I_0$ can't leave the point $\V(x_1,x_2) \subseteq \P^2$,
    \item[$\circ$] the support of $I'_1$ can't leave the line $\V(x_0) \subseteq \P^2$,
    \item[$\circ$] the support of $I'_2$ can't leave the point $\V(x_0,x_1) \subseteq \P^2$
  \end{itemize}
  implies that the locus where $(\Delta_2, \Delta_1, \Delta_0)$ and $(\Delta'_2, \Delta'_1, \Delta'_0)$ 
  intersect is contained in the product of schemes $H^{n_2,\punc}(\A^2)$ parametrizing ideals supported in $(1:0:0)$, 
  $H^{n_1,\punc}(\A^2)$ parametrizing ideals supported in $(0:1:0)$, 
  and $H^{n_0,\punc}(\A^2)$ parametrizing ideals supported in $(0:0:1)$. 
  In particular, if $(I''_2,I''_1,I''_0)$ is a point from the intersection, 
  then 
  \[
    I''_j \in \overline{H^{\Delta_j}_\lex(\A^2)} \cap \overline{H^{(\Delta'_{2-j})^t}_\lex(\A^2)} , 
  \]
  where the closure is taken within $H^{n_2}_\lex(\A^2)$. 

  For $j = 2$, this inclusion says that the component $I''_2$ 
  of any element of $(\Delta_2, \Delta_1, \Delta_0) \cap (\Delta'_2, \Delta'_1, \Delta'_0)$ 
  lies in $(\Delta_2, \emptyset, \emptyset) \cap (\emptyset, \emptyset, \Delta'_0)$. 
  Therefore, $(\Delta_2, \Delta_1, \Delta_0) \cap (\Delta'_2, \Delta'_1, \Delta'_0) \neq \emptyset$ if, and only if, 
  $(\Delta_2, \emptyset, \emptyset) \cap (\emptyset, \emptyset, \Delta'_0) \neq \emptyset$, 
  in which case we write the last two two schemes as closures of BB cells, 
  \[
    \begin{array}{ccccc}
      (\Delta_2, \emptyset, \emptyset) & = & 
      \overline{ H^{\Delta_2}_{w''(2)}(\A^2) \times H^{\emptyset,\lin}_{w''(1)}(\A^2) \times H^{\emptyset,\punc}_{w''(0)}(\A^2) } & = & 
      \BB(w'')_{(\Delta_2,\emptyset, \emptyset)} , \\
      (\emptyset, \emptyset, \Delta'_0) & = & 
      \overline{ H^{(\Delta'_0)^t}_{-w''(0)}(\A^2) \times H^{\emptyset,\lin}_{w''(1)}(\A^2) \times H^{\emptyset,\punc}_{-w''(2)}(\A^2) } & \simeq & 
      \BB(w'')^{((\Delta'_0)^t, \emptyset, \emptyset)} .
    \end{array}
  \]
  Once more Theorem \ref{thm:intersectionOfClosures} allows to derive an inequality on $w''$-weights which, 
  by Lemma \ref{lmm:closedImmersion}, reduces to 
  \[
    - \langle w''(2), \Delta_2 \rangle
    \leq
    - \langle w''(2), (\Delta'_0)^t \rangle ,
  \]
  i.e., the inequality $\langle \mu, \Delta_2 \rangle \geq \langle \mu, (\Delta'_0)^t \rangle$ as claimed in Definition \ref{dfn:orderings}. 
  The claim for $j = 0$ follows by symmetry. 
  For $j = 1$, we analogously see that 
  $(\Delta_2, \Delta_1, \Delta_0) \cap (\Delta'_2, \Delta'_1, \Delta'_0) \neq \emptyset$ if, and only if, 
  $(\emptyset, \Delta_1, \emptyset) \cap (\emptyset, \Delta'_1, \emptyset) \neq \emptyset$, 
  in which case the inequality from Theorem \ref{thm:intersectionOfClosures} reducing to 
  $\langle \lambda, \Delta_1 \rangle \geq \langle \lambda, (\Delta'_1)^t \rangle$. 
\end{proof}


\section*{Appendix}

We first collect a few elementary facts about the natural partial ordering $\leq$ and the partial orderings 
$\leq_\xi$ on $\st_m$ induced by weights $\xi = \mu,\lambda,\nu$ as in Definition \ref{dfn:orderings}, 
only proving the least obvious statement. 

\begin{lmm}
\label{lmm:dualityOfOrderings}
  The two partial orderings $\leq_\mu$ and $\leq_\nu$ on $\st_m$ are dual to each other in the sense that 
  $\Delta \leq_\mu \Delta'$ if, and only if, $(\Delta')^t \leq_\nu \Delta^t$. 
\end{lmm}

\begin{lmm}
\label{lmm:productOrdering}
  The partial ordering $\leq_\mu$ on $\st_m$ is the lexicographic refinement of $\leq_{(-1,0)}$ by $\leq_{(0,-1)}$
  in the sense that $\Delta \leq_\mu \Delta'$ if, and only if, 
  \begin{itemize}
    \item $\Delta \leq_{(-1,0)} \Delta'$, or 
    \item $\Delta =_{(-1,0)} \Delta'$ and $\Delta \leq_{(0,-1)} \Delta'$. 
  \end{itemize}
\end{lmm}

\begin{pro}
  The three partial orderings $\leq_\xi$, for $\xi = \mu, \lambda, \nu$, are refinements of the natural partial ordering $\leq$ on $\st_m$
  in the sense that $\Delta <_\xi \Delta'$ whenever $\Delta < \Delta'$, 
\end{pro}

\begin{proof}
  Assume that the partial ordering $\leq_{(-1,0)}$ is a refinement of the natural partial ordering. 
  Then Lemma \ref{lmm:productOrdering} 
  implies that the statement of the proposition holds for weights $\xi = \mu$ and $\xi = \nu$. 
  Moreover, the identity $\langle \xi, \Delta \rangle = \langle (\xi_1,0), \Delta \rangle + \langle (0,\xi_2), \Delta \rangle$ 
  plus the assumption that $\lambda_1 < 0 < \lambda_2$ plus Lemma \ref{lmm:dualityOfOrderings}
  implies that the statement of the proposition holds for weight $\xi = \lambda$. 
  
  Consider two elements of $\st_m$ such that $\Delta < \Delta'$. 
  We label each box in $\Delta' \setminus \Delta$ by the its row index, 
  and each box in $\Delta \setminus \Delta'$ by the negative of its row index, 
  thus obtaining a disjoint sum of two \defn{skew Young tableaux}, cf. Figure \ref{fig:differences}.
  Then the sum of the labels of all boxes from the tableaux is $\langle (-1,0), \Delta \rangle - \langle (-1,0), \Delta' \rangle$. 
  For showing that the sum is positive, we enumerate the elements of the respective differences in the lex-increasing way, 
  \begin{equation*}
    \begin{split}
      \Delta \setminus \Delta' & = \left\{ \alpha_1, \ldots, \alpha_d \right\} , \\
      \Delta' \setminus \Delta & = \left\{ \alpha'_1, \ldots, \alpha'_d \right\} . 
    \end{split}
  \end{equation*}
  The assumption that $\Delta < \Delta'$ implies that $\alpha_i < \alpha'_i$ for all $i$. 
  Positivity follows. 
\end{proof}

Lastly we give explicit descriptions of the generic monomial ideals from Definition \ref{dfn:genericMonomialIdeals}. 

\begin{center}
\begin{figure}[ht]
  \unitlength0.40mm
  \begin{picture}(170,170)
  \put(0,0){\line(1,0){170}}
  \put(0,0){\line(0,1){170}}
  \multiput(4,133)(10,0){1}{\tiny $0$}
  \multiput(11,133)(0,-10){3}{\tiny $-1$}
  \multiput(24,103)(0,-10){1}{\tiny $2$}
  \multiput(34,103)(0,-10){2}{\tiny $3$}
  \multiput(41,83)(0,-10){2}{\tiny $-4$}
  \multiput(51,83)(0,-10){2}{\tiny $-5$}
  \multiput(61,83)(0,-10){3}{\tiny $-6$}
  \multiput(74,53)(0,-10){2}{\tiny $7$}
  \multiput(84,53)(0,-10){2}{\tiny $8$}
  \multiput(91,33)(0,-10){2}{\tiny $-9$}
  \multiput(102.5,13)(0,-10){2}{\tiny $10$}
  \multiput(112.5,13)(0,-10){2}{\tiny $11$}
  \multiput(122.5,13)(0,-10){2}{\tiny $12$}
  \color{red}
  \put(0,140){\line(1,0){20}}
  \put(20,140){\line(0,-1){40}}
  \put(20,100){\line(1,0){10}}
  \put(30,100){\line(0,-1){10}}
  \put(30,90){\line(1,0){40}}
  \put(70,90){\line(0,-1){50}}
  \put(70,40){\line(1,0){30}}
  \put(100,40){\line(0,-1){40}}
  \color{blue}
  \put(0,130){\line(1,0){10}}
  \put(10,130){\line(0,-1){20}}
  \put(10,110){\line(1,0){30}}
  \put(40,110){\line(0,-1){40}}
  \put(40,70){\line(1,0){20}}
  \put(60,70){\line(0,-1){10}}
  \put(60,60){\line(1,0){30}}
  \put(90,60){\line(0,-1){40}}
  \put(90,20){\line(1,0){40}}
  \put(130,20){\line(0,-1){20}}
  \end{picture}
\caption{The differences of two standard sets ${\color{red}\Delta} \leq {\color{blue}\Delta'}$}
\label{fig:differences}
\end{figure}
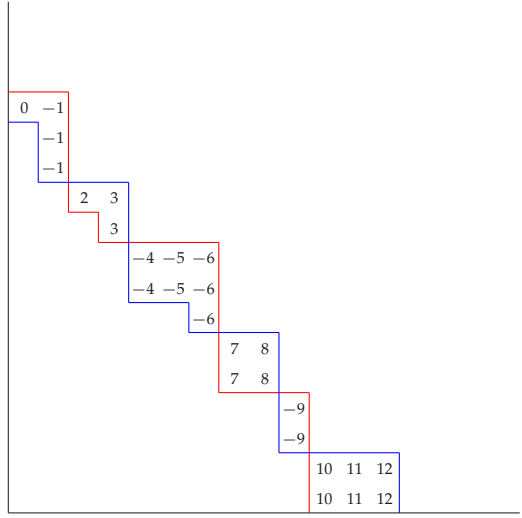
\end{center}

\begin{pro}
  Let $u,v \in \Z^2$ be weights such that $u_1,u_2 < 0$ and $v_1,v_2 > 0$, respectively. 
  \begin{enumerate}[(i)]
    \item The generic monomial ideal in $H^n(\A^2)$ with respect to $u$ is the unique $M_\Gamma$ such that for 
    all $\alpha \in \N^2 \setminus \Gamma$ and for all $\beta \in \Gamma$, the inequality 
    \[
      \langle u, \alpha - \beta \rangle < 0
    \]
    holds true. 
    \item More explicitly, if $u_1 < u_2$, then $\Gamma$ contains the first $n$ members of the sequence 
    \[
      \begin{array}{cccc}
      (0,0), \\
      (0,1), \\
      \vdots \\
      (0,m-1), \\
      (1,0), & (0,m), \\
      (1,1), & (0,m+1), \\
      \vdots \\
      (1,m-1), & (0,2m-1), \\
      (2,0), & (1,m), & (0,2m), \\
      (2,1), & (1,m+1), & (0,2m+1), \\
      \vdots \\
      (2,m-1), & (1,2m-1), & (0,3m-1), & \ldots ,
      \end{array}
    \]
    where $m$ is the unique integer such that $m-1 < u_2 / u_1 < m$. 
    If $u_1 > u_2$, then the transposed analogue of this statement holds true. 
    \item The generic monomial ideal in $H^{n,\punc}(\A^2)$ with respect to $v$ is the 
    vertical strip $\{0\} \times \{0, \ldots, n-1\}$ if $v_1 < v_2$ and the horizontal strip $\{0, \ldots, n-1\} \times \{0\}$ otherwise. 
  \end{enumerate}
\end{pro}

\begin{proof}
  (i) The scheme $H^n(\A^2)$ is covered by affine open patches
  \[
    H^\Delta(\A^2) := \bigl\{ \text{ideals } I \subseteq S' : S' / I \text{ is free with a basis } (x^\beta : \beta \in \Delta) \bigr\}, 
  \]
  one for each standard set $\Delta$ of cardinality $n$ \cite{krbook,huibregtseElementary,norge,strata}. 
  The coordinate ring of that scheme is a quotient of polynomial ring
  \[
    T := \CC[T_{\alpha,\beta} : \alpha \in \widehat{\Delta}] ,
  \]
  where $\widehat{\Delta}$ is a sufficiently large finite standard set containing $\Delta$,  
  by an ideal $J \subseteq T$ spanned by quadratic equations expressing that the quotient of $T[x_1,x_2]$ by the ideal 
  \[
    I := \bigl\langle x^\alpha - \sum_{\beta \in \Delta} T_{\alpha,\beta} x^\beta : \alpha \in \widehat{\Delta} \setminus \Delta \bigr\rangle
  \]
  be a free $T$-module with basis $(x^\beta : \beta \in \Delta)$. 
  In other words, $I$ defines a  $T$-valued point in $H^\Delta(\A^2)$. 
  The BB sink $H^\Delta_u(\A^2)$ is obtained from $H^\Delta(\A^2)$ by killing all $T_{\alpha,\beta}$ 
  such that $\langle u, \alpha - \beta \rangle > 0$. 
  The scheme $H^n(\A^2)$ being smooth of dimension $2n$, so is $H^\Delta(\A^2)$. 
  The same holds true for $H^\Delta_u(\A^2)$ if, and only if, none of the variables $T_{\alpha,\beta}$ 
  such that $\langle u, \alpha - \beta \rangle > 0$ get killed. 
  
  (ii) is elementary. 
  
  (iii) The scheme $H^{n,\punc}(\A^2)$ is covered by affine open patches $H^{\Delta,\punc}(\A^2) := H^\Delta(\A^2) \cap H^{n,\punc}(\A^2)$. 
  The ideal 
  \[
    I := \bigl\langle x_1 - \sum_{j = 1}^{n-1} T_j x_2^j, x_2^n \bigr\rangle ,
  \]
  where the $T_j$ are variables, 
  defines a $\CC[T_1,\ldots,T_n]$-valued point in $H^{n,\punc}(\A^2)$. 
  Therefore, if $\Gamma$ is the vertical strip as defined in the proposition, 
  then $H^{\Gamma,\punc}(\A^2) = \Spec \CC[T_1,\ldots,T_n]$. 
  If $v_1 < v_2$, then $I$ lies in the BB sink of the $T$-action with weight $v$. 
  This proves the one half of (iii), the second following by symmetry. 
\end{proof}

\bibliography{references}
\bibliographystyle{amsalpha}

\end{document}